\documentclass[11pt]{article}

\usepackage{amsmath,amssymb,amsfonts,amsbsy,mathrsfs}

\renewcommand{\subsection}{\subsubsection}

\newtheorem{theorem}{Theorem}[section]
\newtheorem{lemma}{Lemma}[section]
\newtheorem{proposition}{Proposition}[section]
\newenvironment{proof}{\noindent{\bf Proof}.}{\hfill $\Box$ \vspace*{4mm}}
\newtheorem{remark}{Remark}[section]
\newtheorem{definition}{Definition}[section]

\renewcommand\appendix{\par
  \setcounter{section}{0}
   \renewcommand\thesection{Appendix \Alph{section}}
 }

\newcommand{\nt}{|\hspace{-0.7pt}|\hspace{-0.7pt}|}
\newcommand{\Real}{{\rm Re}\,}

\begin{document}

\title{\bf Structural stability of shock waves \\ in 2D compressible elastodynamics}

\author{{\bf Alessandro Morando}\\
DICATAM, Sezione di Matematica, Universit\`a di Brescia \\ Via Valotti, 9, 25133 Brescia, Italy\\
E-mail: alessandro.morando@ing.unibs.it
\and
{\bf Yuri Trakhinin}\\
Sobolev Institute of Mathematics, Koptyug av. 4, 630090 Novosibirsk, Russia\\
and\\
Novosibirsk State University, Pirogova str. 1, 630090 Novosibirsk, Russia\\
E-mail: trakhin@math.nsc.ru
\and
{\bf Paola Trebeschi}\\
DICATAM, Sezione di Matematica, Universit\`a di Brescia \\ Via Valotti, 9, 25133 Brescia, Italy\\
E-mail: paola.trebeschi@ing.unibs.it
}

\date{ }

\maketitle

\begin{abstract}
We study the two-dimensional structural stability of shock waves in a compressible isentropic inviscid elastic fluid in the sense of the local-in-time existence and uniqueness of discontinuous shock front solutions of the equations of compressible elastodynamics in two space dimensions. By the energy method based on a symmetrization of the wave equation and giving an a priori estimate  without loss of derivatives for solutions of the constant coefficients linearized problem we find a condition sufficient for the uniform stability of rectilinear shock waves. Comparing this condition with that for the uniform stability of shock waves in isentropic gas dynamics, we make the conclusion that the elastic force plays stabilizing role. In particular, we show that, as in isentropic gas dynamics, all compressive shock waves are uniformly stable for convex equations of state. Moreover, for some particular deformations (and general equations of state), by the direct test of the uniform Kreiss--Lopatinski condition we show that the stability condition found by the energy method is not only sufficient but also necessary for uniform stability. As is known, uniform stability implies structural stability of corresponding curved shock waves.
\end{abstract}

\section{Introduction}
\label{s1}

We consider the equations of elastodynamics \cite{Daf,Gurt,Jos} governing the  motion of a compressible isentropic inviscid elastic fluid.
We restrict ourself to two-dimensional (2D) elastic flows. Then, the elastodynamics equations read
\begin{equation}\label{7}
\left\{
\begin{array}{l}
 \partial_t\rho  +{\rm div}\, (\rho v )=0, \\
 \partial_t(\rho v ) +{\rm div}\,(\rho v\otimes v  ) + {\nabla}p-{\rm div}\,(\rho FF^{\top})=0,\\
 \partial_t(\rho F_j ) +{\rm div}\,(\rho F_j\otimes v  -  v\otimes  \rho F_j) =0,\quad j=1,2,
\end{array}
\right.
\end{equation}
where $\rho$ is the density, $v\in\mathbb{R}^2$  is the velocity, $F\in \mathbb{M}(2,2)$ is the deformation gradient,  $F_1=(F_{11},F_{21})$ and $F_2=(F_{12},F_{22})$ are the columns of F, and the pressure $p=p(\rho)$ is a smooth function of $\rho$. Moreover, system \eqref{7} is supplemented by the identity ${\rm div}\,(\rho F^{\top})=0$ which is the set of the two divergence constraints
\begin{equation}\label{8}
{\rm div}\,(\rho F_j)=0\quad (j=1,2)
\end{equation}
on initial data, i.e., one can show that if equations \eqref{8} are satisfied initially, then they hold for all $t > 0$. We note that system \eqref{7} arises as the inviscid limit of the equations of compressible viscoelasticity \cite{Daf,Gurt,Jos} of Oldroyd type \cite{Old1,Old2}.

Taking into account the divergence constraints \eqref{8}, we easily symmetrize the system of conservation laws \eqref{7} by rewriting it as
\begin{equation}
\left\{
\begin{array}{l}
{\displaystyle\frac{1}{\rho c^2}\,\frac{{\rm d} p}{{\rm d}t} +{\rm div}\,{v} =0,}\\[6pt]
{\displaystyle\rho\, \frac{{\rm d}v}{{\rm d}t}+{\nabla}p -\rho (F_1\cdot\nabla )F_1 -\rho (F_2\cdot\nabla )F_2=0 ,}\\[6pt]
\rho\,{\displaystyle \frac{{\rm d} F_j}{{\rm d}t}-\rho \,(F_j\cdot\nabla )v =0,} \quad j=1,2,
\end{array}\right. \label{9}
\end{equation}
where $c^2=p'(\rho)$ is the square of the sound speed and ${\rm d} /{\rm d} t =\partial_t+({v} \cdot{\nabla} )$ is the material derivative. Equations \eqref{9} form the symmetric system
\begin{equation}
\label{10}
A_0(U )\partial_tU+A_1(U)\partial_1U +A_2(U)\partial_2U=0
\end{equation}
for $U=(p,v,F_1,F_2)$, where $A_0= {\rm diag} (1/(\rho c^2) ,\rho I_6)$,
\[
A_1=\begin{pmatrix}
{\displaystyle\frac{v_1}{\rho c^2}} & e_1 & \underline{0} & \underline{0}  \\[7pt]
e_1^{\top}&\rho v_1I_2 & -\rho F_{11}I_2 & -\rho F_{12}I_2  \\[3pt]
\underline{0}^{\top} &-\rho F_{11}I_2 & \rho v_1I_2 & O_2 \\
\underline{0}^{\top} &-\rho F_{12}I_2 & O_2 & \rho v_1I_2
\end{pmatrix},\quad
A_2=\begin{pmatrix}
{\displaystyle\frac{v_2}{\rho c^2}} & e_2 & \underline{0} & \underline{0}  \\[7pt]
e_2^{\top}&\rho v_2I_2 & -\rho F_{21}I_2 & -\rho F_{22}I_2  \\[3pt]
\underline{0}^{\top} &-\rho F_{21}I_2 & \rho v_2I_2 & O_2 \\
\underline{0}^{\top} &-\rho F_{22}I_2 & O_2 & \rho v_2I_2
\end{pmatrix},
\]
$e_1=(1,0)$, $e_2=(0,1)$, $\underline{0}=(0,0)$ and  $I_m$ and $O_m$ denote the unit and zero matrices of order $m$ respectively. In \eqref{10} we think of the density as a function of the pressure: $\rho =\rho (p)$, $c^2=1/\rho'(p)$. System \eqref{10} is symmetric hyperbolic if   $A_0>0$, i.e.,
\begin{equation}
\rho >0,\quad  \rho '(p)>0. \label{11}
\end{equation}

The goal of this paper is the study of structural stability of shock waves for the system of hyperbolic conservation laws \eqref{7}. By structural stability we mean the local-in-time existence and uniqueness in Sobolev spaces of shock front solutions. If in the first three scalar equations of system \eqref{7} we set formally $F=0$, then we get the system of isentropic gas dynamics (for the 2D case). The structural stability of shock waves in isentropic gas dynamics was proved by Majda \cite{M2} provided that the uniform stability condition found by him in \cite{M1} is satisfied at each point of the initial shock front. The local-in-time existence and uniqueness theorem in \cite{M2} is proved by the classical fixed-point argument and based on the usage of the linear stability results obtained in \cite{M1} for the linearized problems with constant and variable coefficients.

A priori estimates for the linearized constant coefficients problem were deduced in \cite{M1} by Kreiss' symmetrizers technique \cite{Kreiss,BS} whereas these estimates were extended to the case of variable coefficients by pseudo-differential calculus. It should be noted that in isentropic gas dynamics there are no violently unstable shock waves, i.e., the linearized constant coefficients problem for them always satisfies the Kreiss-Lopatinski condition \cite{Kreiss} (we will below sometimes call it just Lopatinski condition) that is equivalent to the fact that Hadamard-type ill-posedness examples cannot be constructed for this problem. At the same time, a priori estimates without loss of derivatives from the data to the solutions can be derived only for the case when the uniform Kreiss--Lopatinski condition \cite{Kreiss} holds. For shock waves this condition is called uniform stability condition and shock waves for a corresponding domain of parameters of the basic state (``unperturbed flow'') defined by this condition are called {\it uniformly stable}.

Clearly, only a priori estimates without loss of derivatives are, in general, suitable for their usage in the proof of the existence of solutions of the original nonlinear problem by the fixed-point argument. For the case when the Kreiss-Lopatinski condition is satisfied in a weak sense (we have weak or {\it neutral stability}), i.e., the uniform Kreiss-Lopatinski condition fails, only an a priori estimate with a loss of one tangential derivative was obtained in \cite{M1} for the linearized constant coefficients problem for shock waves in isentropic gas dynamics.

Regarding shock waves in full (non-isentropic) gas dynamics, for uniformly stable shocks the well-posedness of the linearized constant coefficients problem was first proved by Blokhin \cite{Bl79} (see also \cite{BThand}) by the energy method based on a symmetrization of the wave equation for the perturbation of the pressure. It should be noted that the domains of uniform stability, neutral stability and violent instability of gas dynamical shock waves were found by D'iakov \cite{D'iak54} by the normal modes analysis (without referring to the terminology connected with weak and uniform Lopatinski conditions).

The structural stability of uniformly stable gas dynamical shock waves was independently proved by Blokhin \cite{Bl81,Bl82} (see also \cite{BThand}) and Majda \cite{M2}. Moreover, Majda has also proved the structural stability of uniformly stable shock waves for abstract hyperbolic symmetrizable systems of conservation laws satisfying some block structure condition \cite{M2}. The results of Majda were then clarified and improved by M\'etivier \cite{Met} who, in particular, essentially relaxed Majda's assumptions about the smoothness of the initial data.

M\'etivier and Zumbrun \cite{MZ} have later extended Majda's structural stability results to hyperbolic symmetrizable systems with characteristics of variable multiplicities for which Majda's block structure condition fails. Several hypotheses introduced by M\'etivier and Zumbrun seem to be satisfied for a wide class of systems, in particular, for the system of ideal compressible magnetohydrodynamics (MHD) \cite{MZ,Kwon}. In MHD there are two types of Lax shocks \cite{Lax}: fast and slow shock waves. The linear stability of fast MHD shock  waves was analyzed in \cite{GK,BD,BTejm,BThand} for some particular cases. A complete 2D stability analysis of fast MHD shock waves was carried out in \cite{Tcmp} for a polytropic gas equation of state. Taking into account the results in \cite{MZ,Kwon}, uniformly stable fast shock waves found in \cite{Tcmp} are structurally stable. Regarding slow MHD shock waves, some results about their stability  can be found in \cite{LessDesh67,Fil,BThand,FT,FKS,BMZ}.

In general, the question about structural stability of neutrally stable shock waves is still an open problem. However, the local-in-time existence of neutrally stable shock waves in isentropic gas dynamics was proved by Coulombel and Secchi \cite{CS} by suitable Nash-Moser iterations. At the same time, it is still unclear whether, for example,  neutrally stable shock waves in full gas dynamics or neutrally stable MHD shocks found in \cite{BTejm,Tcmp} even for a polytropic gas equation of state are always structurally stable.

Returning to elastodynamics, we note that according to our knowledge there were no studies of shock waves for system \eqref{7}. In this connection, we can only mention characteristic discontinuities, namely, vortex sheets for system \eqref{7} whose linear and structural stability was recently studied by Chen, Hu and Wang \cite{CHW1,CHW2,Hu}.

We first show that, as in gas dynamics (see, e.g., \cite{BS,BThand,M1,M2,Met}), in elastodynamics all Lax shock waves are extreme shocks (see Section \ref{s2} where we formulate the free boundary problem for shock waves in elastic fluids). In Section \ref{s3}, we write down the constant coefficients linearized problem associated with rectilinear shock waves. The main difficulty appearing for shock waves in elastodynamics (even for the 2D case) is that, unlike isentropic gas dynamics \cite{M1} or even full gas dynamics \cite{D'iak54,M1}, it is impossible to perform a complete spectral analysis of the constant coefficients problem analytically, i.e., to find analytically the boundaries between the domains of uniform stability, neutral stability and violent instability. At the same time, even for the 2D case considered in this paper, the constant coefficients of the linear problem depend on {\it six} dimensionless parameters. Therefore, a complete numerical test of the Lopatinski condition like that performed in \cite{Tcmp} for fast MHD shock waves seems to be hardly realizable in practice. We overcome this difficulty (for the most part at least) by the energy method proposed by Blokhin \cite{Bl79}. The crucial point is that from the linearized system of elastodynamics with constant coefficients it is also possible to deduce a separate second-order equation for the perturbation of the pressure whose canonical form is the wave equation (see Section \ref{s4}). That is, for elastodynamics Blokhin's energy method \cite{Bl79} is, in some sense, more efficient than spectral analysis.

In Section \ref{s4}, by the energy method we deduce an a priori estimate  without loss of derivatives for solutions of the constant coefficients linearized problem under a condition sufficient for the uniform stability of rectilinear shock waves (see Theorem \ref{t1}). Comparing this condition with Majda's uniform stability condition \cite{M1} for shock waves in isentropic gas dynamics, we make the conclusion that the elastic force plays {\it stabilizing role}. Moreover, we show that, as in isentropic gas dynamics, all compressive shock waves are uniformly stable for convex equations of state. In particular, in a polytropic elastic fluid  all shock waves are structurally stable (see Theorem \ref{t2}).

Clearly, to show that the condition \eqref{usc'} in Theorem \ref{t1} is not only sufficient but also necessary for uniform (and so structural) stability we should show by spectral analysis that its violation leads to the violation of the uniform Lopatinski condition. At least for a part of a generically unbounded domain of six dimensionless parameters, this could be done by using the algorithm of numerical testing of the (uniform) Lopatinski condition proposed in \cite{Tcmp} for extreme shocks. This is postponed to future research. But, for particular deformations (and general equations of state) we still can do something analytically. In Section \ref{s5}, for the case of stretching and a ``symmetric'' case (when the diagonal elements of $F$ are zeros), by the direct test of the uniform Kreiss--Lopatinski condition we show that the stability condition \eqref{usc'} found by the energy method is not only sufficient but also {\it necessary} for uniform stability. We also show that for the particular deformations dealt with in Section \ref{s5} there are no violently unstable shock waves. Our spectral arguments in Section \ref{s5} are based on the definitions of the Lopatinski condition and the uniform Lopatinski condition given in \cite{Tcmp} for extreme shocks (they are equivalent to Kreiss' classical  definitions \cite{Kreiss}).

For convex equations of state (in particular, for a polytropic elastic fluid), we have thus completed the problem on the 2D structural stability of shock waves in isentropic elastic fluids. The main question left unsolved for  general equations of state is that whether shock waves can be violently unstable. Since, as we have shown, the elastic force plays stabilizing role for uniform stability, it is natural to suppose that the same is true for neutral stability, i.e., as in isentropic gas dynamics, there are no violently unstable shocks. However, at the present moment we see only one way to prove this, at least, for a part of the whole unbounded parameter domain of the constant coefficients problem. This is the numerical test of the Lopatinski condition based on the algorithm from \cite{Tcmp}. This is also postponed to future research.

\section{Statement of the free boundary problem for shock waves}
\label{s2}

Let $\Gamma (t)=\{ x_1=\varphi (t,x_2)\}$ be a curve of strong discontinuity for system \eqref{7}, i.e., we are interested in solutions of \eqref{7} that are smooth on either side of $\Gamma (t)$. As is known, to be weak solutions such piecewise smooth solutions should satisfy corresponding Rankine-Hugoniot jump conditions at each point of $\Gamma (t)$.  For the conservation laws \eqref{7} these jump conditions can be written in the following form:
\begin{align}
& [\mathfrak{m}]=0, \label{RH1}\\
& \mathfrak{m}[v_{\rm N}] +\left( 1+(\partial_2\varphi)^2 \right)[p]=\left[\rho (F_{1{\rm N}}^2+ F_{2{\rm N}}^2)\right],\label{RH2}\\
& \mathfrak{m}[v_{\tau}]=\left[\rho (F_{1{\rm N}}F_{1\tau}+ F_{2{\rm N}}F_{2\tau})\right],\label{RH3}\\
& \mathfrak{m}[F_{j{\rm N}}]=[\rho v_{\rm N}F_{j{\rm N}}],\quad j=1,2, \label{RH4}\\
& \mathfrak{m}[F_{j\tau}]=[\rho v_{\tau}F_{j{\rm N}}],\label{RH5}\quad j=1,2,\\
& [\rho F_{j{\rm N}}]=0,\quad j=1,2,\label{RH6}
\end{align}
where $[g]=g^+|_{\Gamma}-g^-|_{\Gamma}$ denotes the jump of $g$, with $g^{\pm}:=g$ in the domains
\[
\Omega^{\pm}(t)=\{\pm (x_1- \varphi (t,x_2))>0\},
\]
and
\[
\mathfrak{m}^{\pm}=\rho^{\pm} (v_{\rm N}^{\pm}-\partial_t\varphi),\quad v_{\rm N}^{\pm}=v_1^{\pm}-v_2^{\pm}\partial_2\varphi ,\quad
F_{j{\rm N}}^{\pm}=F_{1j}-F_{2j}\partial_2\varphi ,
\]
\[
v_{\tau}=v_1\partial_2\varphi +v_2,\quad F_{j\tau}=F_{1j}\partial_2\varphi +F_{2j};
\]
$\mathfrak{m}:=\mathfrak{m}^{\pm}|_{\Gamma}$ is the mass transfer flux across the discontinuity curve. Conditions \eqref{RH6} come actually from the constraint equations \eqref{8}. On the other hand, conditions \eqref{RH4} are rewritten as $\partial_t\varphi [\rho F_{j{\rm N}}]=0$. That is, conditions \eqref{RH4} are implied by \eqref{RH6} and can be thus excluded from system \eqref{RH1}--\eqref{RH6}.

We are interested in {\it shock waves}. For them, as in gas dynamics (see, e.g., \cite{BS,BThand,M1,M2,Met}), $\mathfrak{m}\neq 0$ and $[\rho ]\neq 0$. In view of \eqref{RH6}, conditions \eqref{RH3} and \eqref{RH5} form the following linear algebraic system for the jumps $[v_{\tau}]$, $[F_{1\tau}]$ and $[F_{2\tau}]$:
\begin{equation}\label{m0}
\begin{pmatrix}
\mathfrak{m} & -\rho^+ F_{1{\rm N}}^+ & -\rho^+ F_{2{\rm N}}^+ \\
-\rho^+ F_{1{\rm N}}^+ & \mathfrak{m} & 0 \\
-\rho^+ F_{2{\rm N}}^+ & 0 & \mathfrak{m}
\end{pmatrix}
\begin{pmatrix} [v_{\tau}]\\ [F_{1\tau}]\\ [F_{2\tau}]\end{pmatrix} =0\quad\mbox{on}\ \Gamma .
\end{equation}
Since $\mathfrak{m}\neq 0$, this system has a nonzero solution provided that $\mathfrak{m}^2=(\rho^+)^2\left.\left((F_{1{\rm N}}^+)^2+(F_{2{\rm N}}^+)^2 \right)\right|_{\Gamma}$. We assume that $\mathfrak{m}^2\neq(\rho^+)^2\left.\left((F_{1{\rm N}}^+)^2+(F_{2{\rm N}}^+)^2 \right)\right|_{\Gamma}\,$, i.e.,
\begin{equation}\label{m}
(v_{\rm N}^+-\partial_t\varphi)^2\neq (F_{1{\rm N}}^+)^2+(F_{2{\rm N}}^+)^2 \quad\mbox{on}\ \Gamma .
\end{equation}
In fact, by virtue of \eqref{RH1} and \eqref{RH6}, assumption \eqref{m} can be also written as $(v_{\rm N}^--\partial_t\varphi)^2\neq (F_{1{\rm N}}^-)^2+(F_{2{\rm N}}^-)^2$ on $\Gamma$. As we will see below, \eqref{m} holds thanks to the Lax conditions \cite{Lax}.

It follows from \eqref{m0} and \eqref{m} that
\[
[v_{\tau}]=0,\quad  [F_{1\tau}]=0,\quad [F_{2\tau}]=0.
\]
In particular, we see that, as for gas dynamical shock waves, the tangential component of the velocity is continuous on the shock front ($[v_{\tau}]=0$). In view of  \eqref{RH1} and \eqref{RH6}, we can also rewrite condition \eqref{RH2} as
\[
\mathfrak{M}[V]+\left( 1+(\partial_2\varphi)^2 \right)[p]=0,
\]
where $\mathfrak{M}=\mathfrak{m}^2-(\rho^+)^2\left.\left((F_{1{\rm N}}^+)^2+(F_{2{\rm N}}^+)^2 \right)\right|_{\Gamma}\neq 0$, cf. \eqref{m}, and $V=1/\rho$ is the specific volume (we assume that we are away from vacuum, cf. \eqref{11}). We thus have the following seven boundary conditions on the curve $\Gamma (t)$ of a shock wave:
\begin{equation}\label{bc}
[\mathfrak{m}]=0,\quad \mathfrak{M}[V]+\left( 1+(\partial_2\varphi)^2 \right)[p]=0,\quad [v_{\tau}]=0,\quad  [F_{j\tau}]=0,\quad [\rho F_{j{\rm N}}]=0,\quad j=1,2.
\end{equation}

The free boundary problem for shock waves is the problem for the systems
\begin{equation}
A_0(U^{\pm})\partial_tU^{\pm}+A_1(U^{\pm} )\partial_1U^{\pm}+A_2(U^{\pm} )\partial_2U^{\pm}=0\quad \mbox{in}\ \Omega^{\pm}(t),
\label{21}
\end{equation}
cf. \eqref{10}, with the boundary conditions \eqref{bc} on $\Gamma (t)$ and  the initial
\begin{equation}
{U}^{\pm} (0,{x})={U}_0^{\pm}({x}),\quad {x}\in \Omega^{\pm} (0),\quad \varphi (0,{x}_2)=\varphi _0({x}_2),\quad {x}_2\in\mathbb{R}.\label{indat}
\end{equation}
For problem \eqref{bc}--\eqref{indat} we should assume $\det F^{\pm}|_{t=0}\neq 0$ in $\Omega^{\pm} (0)$. Then, for smooth solutions $\det F^{\pm}\neq 0$ in $\Omega^{\pm} (t)$ on a short time interval (see also Remark \ref{det} below). Moreover, as for the Cauchy problem, the divergence constraints \eqref{8} are preserved by problem \eqref{bc}--\eqref{indat}.

\begin{proposition}
Suppose that problem \eqref{bc}--\eqref{indat} has a smooth solution $(U^+,U^-,\varphi)$ for $t\in [0,T]$ satisfying the shock wave assumption $\mathfrak{m}\neq 0$. Then, if the initial data \eqref{indat} satisfy \eqref{8} in $\Omega^{\pm} (0)$, then
\begin{equation}\label{8'}
{\rm div}\,(\rho^{\pm} F_j^{\pm})=0\quad \mbox{in}\ \Omega^{\pm}(t) \quad (j=1,2)
\end{equation}
for all $t\in [0,T]$.
\label{p1}
\end{proposition}

\begin{proof}
Since $\mathfrak{m}\neq 0$, without loss of generality we may suppose that $v_{\rm N}^{\pm}|_{\Gamma}>\partial_t\varphi$.
It follows from the first and the last four equations of systems \eqref{21} that
\begin{equation}
\partial_t(\rho^{\pm} F_j^{\pm})+ {\rm curl}\, (\rho^{\pm} F_j^{\pm} \times v^{\pm} ) + {\rm div}\,(\rho^{\pm} F_j^{\pm})\,v^{\pm} =0\quad \mbox{in}\ \Omega^{\pm}(t).
\label{22}
\end{equation}
Using then \eqref{22} and $v_{\rm N}^{\pm}|_{\Gamma}>\partial_t\varphi$ and following literally the arguments from \cite{MZ} towards the proof of the divergence constraint ${\rm div}\, H=0$ for the magnetic field $H$ on both side of the MHD shock front, we get constraints \eqref{8'} for all $t\in [0,T]$.
\end{proof}

We can reduce the free boundary problem  \eqref{bc}--\eqref{indat} to that in the fixed domains $\mathbb{R}^2_{\pm}=\{\pm x_1>0,\ x_2\in\mathbb{R}\}$ by the simple change of variables
\begin{equation}
{x}'_1=x_1-\varphi (t,x_2).
\label{chv}
\end{equation}
Dropping primes, from systems \eqref{21} we obtain
\begin{equation}
A_0(U^{\pm})\partial_tU^{\pm}+\widetilde{A}_{1}(U^{\pm},\varphi )\partial_1U^{\pm}+A_2(U^{\pm} )\partial_2U^{\pm}=0\quad \mbox{for}\ x\in \mathbb{R}^2_{\pm},
\label{23}
\end{equation}
where $\tilde{A}_{1}$ is the so-called boundary matrix:
\[
\widetilde{A}_{1}=\widetilde{A}_{1}(U,\varphi )= A_1(U)-A_0(U)\partial_t\varphi -A_2(U)\partial_2\varphi  .
\]
The boundary conditions for \eqref{23} are \eqref{bc} on the line $x_1=0$.

It is well-known that the necessary condition for the well-posedness of the above problem is that the number of boundary conditions should be one unit greater than the number of incoming characteristics of the 1D counterparts (with $A_2=0$) of systems \eqref{23} for fixed (``frozen'') $U^{\pm}$ and $\varphi$ satisfying the boundary conditions (roughly speaking, one of the boundary conditions is needed for finding the unknown function $\varphi$). Since the number of incoming characteristics is defined by the number of positive/negative eigenvalues of the matrices $\left(A_0(U^{\pm})\right)^{-1}\widetilde{A}_{1}(U^{\pm},\varphi )|_{x_1=0}$, for shock waves (for them these matrices have no zero eigenvalues) this is equivalent to the Lax's $k$--shock conditions \cite{Lax}
\[
\lambda_{k-1}^- <\partial_t\varphi <\lambda_k^-,\quad \lambda_k^+ <\partial_t\varphi <\lambda_{k+1}^+
\]
for some fixed integer $k$, where for our case of system of seven equations $1\leq k\leq 7$ and $\lambda_j^\pm$ ($j=\overline{1,7}$) are the eigenvalues of the matrices
\[
A_{\rm N}^{\pm}:=\left(A_0(U^{\pm})\right)^{-1}\left( A_1(U^{\pm})-A_2(U^{\pm})\partial_2\varphi  \right)|_{x_1=0},
\]
with $U^{\pm}$ and $\varphi$ satisfying the boundary conditions \eqref{bc} on $x_1=0$. Moreover, $\lambda_j^\pm$ are numbered as
\[
\lambda_1^-\leq \ldots \leq\lambda_7^-,\quad \lambda_1^+\leq \ldots \leq\lambda_7^+,
\]
and we take $\lambda^-_0:=-|\partial_t\varphi |/2$ and $\lambda_8^+:=2|\partial_t\varphi |$.

By direct calculation we find the eigenvalues $\lambda_j^\pm$:
\begin{equation}
\left\{
\begin{array}{l}
 \lambda_1^{\pm}=v_{\rm N}^{\pm}-\sqrt{(c^{\pm})^2+(F_{1{\rm N}}^{\pm})^2+(F_{2{\rm N}}^{\pm})^2}\,,\quad \lambda_2^{\pm}=v_{\rm N}^{\pm}-\sqrt{(F_{1{\rm N}}^{\pm})^2+(F_{2{\rm N}}^{\pm})^2},\\[3pt]  \lambda_3^{\pm}=\lambda_4^{\pm}=\lambda_5^{\pm}=v_{\rm N}^{\pm}\,,\\[3pt]
 \lambda_6^{\pm}=v_{\rm N}^{\pm}+\sqrt{(F_{1{\rm N}}^{\pm})^2+(F_{2{\rm N}}^{\pm})^2}\,,\quad \lambda_7^{\pm}=v_{\rm N}^{\pm}+\sqrt{(c^{\pm})^2+(F_{1{\rm N}}^{\pm})^2+(F_{2{\rm N}}^{\pm})^2}\quad \mbox{on}\ x_1=0,
\end{array}
\right.
\label{24}
\end{equation}
where $c^{\pm} =1/\sqrt{\rho'(p^{\pm})}$ are the sound speeds ahead and behind of the shock. Without loss of generality we assume that $v_{\rm N}^{\pm}|_{x_1=0}>\partial_t\varphi$. Then, $k$--shocks with $k\geq 3$ are not realizable. By virtue of the first and the last two boundary conditions in \eqref{bc}, the inequalities $\lambda_2^->\partial_t\varphi$ and $\lambda_2^+<\partial_t\varphi$ appearing for $k=2$ contradict each other. That is, as in gas dynamics, only 1-shocks are possible:
\begin{equation}
\left\{
\begin{array}{l}
v_{\rm N}^--\partial_t\varphi >\sqrt{(c^-)^2+(F_{1{\rm N}}^{-})^2+(F_{2{\rm N}}^{-})^2}\,,\\[6pt]
 \sqrt{(F_{1{\rm N}}^{+})^2+(F_{2{\rm N}}^{+})^2}<v_N^+-\partial_t\varphi <\sqrt{(c^{+})^2+(F_{1{\rm N}}^{+})^2+(F_{2{\rm N}}^{+})^2}\quad
 \mbox{on}\ x_1=0.
\end{array}
\right.
\label{25}
\end{equation}
We note that 1--shocks are {\it extreme shocks} in the sense that ahead of the shock there are no incoming waves ($\lambda_j^->\partial_t\varphi $ for all $j=\overline{1,7}$). We see that the Lax conditions \eqref{25} guarantee the fulfilment of our assumption \eqref{m}.

\begin{remark}
{\rm
The second inequality in \eqref{25} implies $\mathfrak{M}>0$. Then, it follows from the second condition in \eqref{bc} that the compressibility conditions $[p]>0$ and $[\rho]>0$ are equivalent (exactly as for gas dynamical shock waves). At the same time, rarefaction 1-shocks are, in general, possible, as in isentropic gas dynamics. Recall that in full (non-isentropic) gas dynamics rarefaction Lax shocks are commonly (at least, for so-called normal gases) excluded by the physical entropy increase condition (see, e.g., \cite{BThand} and references therein). However, it should be also noted that in our case of isentropic elastodynamics, as in isentropic gas dynamics, the density balance $[\rho ]=0$ is prohibited by the Lax conditions \eqref{bc}. Indeed, $[\rho]=0$ implies $[p]=0$ and $[c]=0$. But, in view of the last two boundary conditions in \eqref{bc}, we see that in this case the first and third inequalities in \eqref{25} contradict each other.
}
\label{r1}
\end{remark}

\begin{remark}
{\rm
Clearly, for systems \eqref{23} with the boundary conditions \eqref{bc} on $x_1=0$ the ``curved'' divergence constraints \eqref{8'} take place:
\begin{equation}
{\rm div}\,(\rho^{\pm} \widetilde{F}_j^{\pm})=0\quad \mbox{in}\ \mathbb{R}^2_{\pm} \quad (j=1,2),
\label{8"}
\end{equation}
where $\widetilde{F}_j^{\pm}=(F_{j{\rm N}},F_{2j})$. In fact, we can prove \eqref{8"} without reference to Proposition \ref{p1} by using arguments similar to those from \cite{T09} applied to the problem for current-vortex sheets.
}
\end{remark}

\begin{remark}
{\rm
In view of the requirements $\det F^{\pm}\neq 0$, the expressions $(F_{1{\rm N}}^{\pm})^2+(F_{2{\rm N}}^{\pm})^2$ on the shock front appearing in \eqref{24} and \eqref{25} may not vanish: $(F_{1{\rm N}}^{\pm})^2+(F_{2{\rm N}}^{\pm})^2\neq 0$ on $\Gamma (t)$ (or at $x_1=0$, as in \eqref{24} and \eqref{25}, if we have already straightened the curved shock). As was already noted above, for smooth shock front solutions (smooth from both sides of the shock) the assumption $|\det F^{\pm}_{|t=0}|\geq \epsilon >0$ in $\Omega^{\pm} (0)$ on the initial data can guarantee that $\det F^{\pm}\neq 0$ in $\Omega^{\pm} (t)$ on a short time interval. At the same time, for vortex sheets studied in \cite{CHW1,CHW2,Hu} we have $F_{1{\rm N}}^{\pm}|_{\Gamma}=F_{2{\rm N}}^{\pm}|_{\Gamma}=0$ by definition of vortex sheets. This leads to the degeneracy $\det F^{\pm}_{|\Gamma}= 0$. As was noticed in \cite{CHW2}, the associated physical interpretation
of this degeneracy is still unclear. The same physical puzzle appears for the free boundary problem in compressible \cite{T18} and incompressible elastodynamics \cite{HW,LiWangZhang} with boundary conditions requiring that all columns of the deformation gradient are parallel to the boundary (in the compressible case the situation is better because we may have $\det F\neq 0$ inside the flow domain whereas $\det F (t,x)\rightarrow 0$ as $x$ tends to a point of the boundary $\Gamma$). Fortunately, for shock waves we have no such a degeneracy.
}
\label{det}
\end{remark}

\section{Constant coefficients linearized problem associated with recti\-linear shock waves}
\label{s3}

Without loss of generality we can consider the unperturbed rectilinear shock wave with the equation $x_1=0$. We consider a constant solution
$(U^+,U^-,\varphi) =(\widehat{U}^+,\widehat{U}^-,\hat{\varphi})$ of systems \eqref{23} and the boundary conditions \eqref{bc} associated with this shock wave:
\[
\widehat{U}^{\pm}=(\hat{p}^{\pm},\hat{v}^{\pm},\widehat{F}_1^{\pm},\widehat{F}_2^{\pm}), \quad \hat{\varphi} =0,\quad \hat{\rho}^{\pm}=\rho (\hat{p}^{\pm}),\quad \hat{c}^{\pm}=1/\sqrt{\rho'(\hat{p}^{\pm})},
\]
\[
\hat{v}^{\pm}=(\hat{v}_1^{\pm},\hat{v}_2^{\pm}),\quad
\widehat{F}_1^{\pm} =(\widehat{F}_{11}^{\pm},\widehat{F}_{21}^{\pm}), \quad \widehat{F}_2^{\pm} =(\widehat{F}_{12}^{\pm},\widehat{F}_{22}^{\pm}),
\]
where all the hat values are given constants. We assume that $\hat{v}_1^{\pm}>0$. In view of the third condition in \eqref{bc}, $\hat{v}_2^+=\hat{v}_2^-$ and we can choose a reference frame in which
$\hat{v}_2^+=\hat{v}_2^-=0$. The rest constants satisfy the relations
\begin{equation}
\frac{\hat{\rho}^+}{\hat{\rho}^-}=\frac{\hat{v}_1^-}{\hat{v}_1^+},\quad \frac{\hat{\rho}^+}{\hat{\rho}^-}\big\{ (\hat{v}_1^+)^2-\big((\widehat{F}_{11}^+)^2+(\widehat{F}_{12}^+)^2\big)\big\}[\hat{\rho}]=[\hat{p}],\quad \big[\widehat{F}_{2j}\big]=0,\quad \big[\hat{\rho}\widehat{F}_{1j}\big]=0
\label{sbc}
\end{equation}
following from \eqref{bc}, where $j=1,2$, $[\hat{\rho}]=\hat{\rho}^+-\hat{\rho}^-$, etc.

We also assume that the constant solution satisfies the Lax conditions \eqref{25}. For the constant solution behind of the shock they read:
\begin{equation}
M_1<M<M_*,
\label{Mach}
\end{equation}
where $M={\hat{v}_1^+}/{\hat{c}^+}$ is the downstream  Mach number,
\[
M_1=\sqrt{\mathcal{F}_{11}^2+\mathcal{F}_{12}^2},\quad M_*=\sqrt{1+\mathcal{F}_{11}^2+\mathcal{F}_{12}^2},
\]
and $\mathcal{F}_{ij}=\widehat{F}_{ij}^+/\hat{c}^+$ are the components of the unperturbed scaled deformation gradient $\mathcal{F}=(\mathcal{F}_{ij})_{i,j=1,2}$ behind of the shock. Introducing the upstream Mach number $M_-={\hat{v}_1^-}/{\hat{c}^-}$ and using the jump conditions \eqref{sbc}, for the constant solution we also write down  the first Lax condition in \eqref{25}:
\begin{equation}
M_->\frac{M}{\sqrt{M^2-M_1^2}}.
\label{Mach-}
\end{equation}

In view of the first inequality in \eqref{25} written for the constant solution ahead of the shock, all the characteristics of the linear system
\[
A_0(\widehat{U}^-)\partial_t(\delta U^-)+\widetilde{A}_{1}(\widehat{U}^-,0 )\partial_1(\delta U^-)+A_2(\widehat{U}^- )\partial_2(\delta U^-)=0\quad \mbox{for}\ x\in \mathbb{R}^2_-
\]
for the perturbation $\delta U^-$ ahead of the shock are outgoing (as we already noticed above, our shock wave is an extreme shock), i.e., this linear system does not need any any boundary conditions. Hence, as is known, without loss of generality we may assume that $\delta U^-= 0$. Linearizing system \eqref{23} behind of the shock about the constant solution, we obtain the linear constant coefficients system
\begin{equation}
A_0(\widehat{U}^+)\partial_t(\delta U^+)+\widetilde{A}_{1}(\widehat{U}^+,0 )\partial_1(\delta U^+)+A_2(\widehat{U}^+ )\partial_2(\delta U^+)=0\quad \mbox{for}\ x\in \mathbb{R}^2_+
\label{ls}
\end{equation}
for the perturbation $\delta U^+=(\delta p^+,\delta v^+,\delta F_1^+,\delta F_2^+)$, where $\delta v^+=(\delta v^+_1,\delta v^+_2)$, etc.

For the forthcoming analysis of the linearized problem, it is convenient to reduce it to a dimensionless form by introducing the following scaled values:
\[
x'=\frac{x}{l},\quad t'=\frac{\hat{v}_1^+t}{l},\quad p=\frac{\delta p^+}{\hat{\rho}^+(\hat{c}^+)^2},\quad v=\frac{\delta v^+}{\hat{v}_1^+},\quad F_{ij}=\frac{\delta F_{ij}^+}{\hat{c}^+},\quad  \varphi =\frac{\delta\varphi}{l},
\]
where $l$ is a typical length and $\delta\varphi$ is the perturbation of the rectilinear shock wave. Recall that $\hat{v}_1^+>0$. We also assume that $\hat{c}^+>0$, cf. \eqref{11}. After dropping the primes and taking into account that $\hat{v}_2^+=0$, system \eqref{ls} is rewritten as the following linear system for the scaled perturbation $U=(p,v,F_1,F_2)$ behind of the shock wave:
\begin{alignat}{3}
&Lp+{\rm div}\,v=0, &  \label{ls1}\\
& M^2Lv+\nabla p -(\mathcal{F}_1\cdot\nabla )F_1-(\mathcal{F}_2\cdot\nabla )F_2=0, & \label{ls2}\\
& LF_1- (\mathcal{F}_1\cdot\nabla )v=0,&  \label{ls3}\\
&  LF_2- (\mathcal{F}_2\cdot\nabla )v=0 &\quad \mbox{for}\ x\in \mathbb{R}^2_+,\label{ls4}
\end{alignat}
where $L =\partial_t+\partial_1$, $\mathcal{F}_1=(\mathcal{F}_{11},\mathcal{F}_{21})$ and $\mathcal{F}_2=(\mathcal{F}_{12},\mathcal{F}_{22})$. For more technical simplicity of forthcoming arguments, here we do not introduce source terms (given right-hand parts) and consider homogeneous interior equations. For the forthcoming analysis it will be also useful to rewrite system \eqref{ls1}--\eqref{ls4} in the matrix form
\begin{equation}\label{lsm}
\mathcal{A}_0\partial_tU+\mathcal{A}_1\partial_1U+\mathcal{A}_2\partial_2U=0\quad\mbox{for}\ x\in \mathbb{R}^2_+,
\end{equation}
where $\mathcal{A}_0={\rm diag}\,(1,M^2I_2,I_4)$,
\[
\mathcal{A}_1=\begin{pmatrix}
1 & e_1 & \underline{0} & \underline{0}  \\[7pt]
e_1^{\top}&M^2I_2 &  -\mathcal{F}_{11}I_2 & - \mathcal{F}_{12}I_2  \\[3pt]
\underline{0}^{\top} &- \mathcal{F}_{11}I_2 & I_2 & O_2 \\
\underline{0}^{\top} &- \mathcal{F}_{12}I_2 & O_2 & I_2
\end{pmatrix},\quad
\mathcal{A}_2=\begin{pmatrix}
0 & e_2 & \underline{0} & \underline{0}  \\[7pt]
e_2^{\top}&O_2 & - \mathcal{F}_{21}I_2 & - \mathcal{F}_{22}I_2  \\[3pt]
\underline{0}^{\top} &- \mathcal{F}_{21}I_2 & O_2 & O_2 \\
\underline{0}^{\top} &- \mathcal{F}_{22}I_2 & O_2 & O_2
\end{pmatrix}.
\]

We now linearize the boundary conditions \eqref{bc} on $x_1=0$. Taking into account \eqref{sbc}, $\delta U^-= 0$ and omitting technical calculations, we get the following linearized boundary conditions for system \eqref{ls1}--\eqref{ls4} written in a dimensionless form:
\begin{equation}\label{lbc}
\left\{
\begin{array}{ll}
{\displaystyle v_1 +d_0p -\frac{\ell_0}{M^2R}\,v_2=0}, &\\[9pt]
 a_0p+(1-R)\partial_{\star}\varphi =0,\qquad
v_2 +(1-R)\partial_2\varphi= 0, & \\[9pt]
{\displaystyle F_{11}+{\cal F}_{11}\,p-\frac{{\cal F}_{21}}{R}\,v_2=0,} \qquad
{\displaystyle F_{12}+{\cal F}_{12}\,p-\frac{{\cal F}_{22}}{R}\,v_2=0,}  & \\[9pt]
F_{21}-{\cal F}_{11}\,v_2=0, \qquad   F_{22}-{\cal F}_{12}\,v_2=0 & \ \mbox{on}\ x_1=0,
\end{array}
\right.
\end{equation}
where
\[
d_0=\frac{M_*^2+M^2}{2M^2},\quad R=\frac{\hat{\rho}^+}{\hat{\rho}^-},\quad \ell_0={\cal F}_{11}{\cal F}_{21}+{\cal F}_{12}{\cal F}_{22},\quad  a_0=-\frac{\beta^2R}{2M^2},
\]
\[
\beta =\sqrt{M_*^2-M^2}\quad (\mbox{cf.}\ \eqref{Mach}),\quad \partial_{\star}=\partial_t-\frac{\ell_0}{M^2}\,\partial_2.
\]
We note that $R\neq 1$ (see Remark \ref{r1}).
Again for technical simplicity, we consider homogeneous boundary conditions. Thus, our constant coefficients linearized problem is \eqref{lsm}, \eqref{lbc} with the initial data
\begin{equation}
{U} (0,{x})={U}_0({x}),\quad {x}\in \mathbb{R}^2,\quad \varphi (0,{x}_2)=\varphi _0({x}_2),\quad {x}_2\in\mathbb{R}.\label{lindat}
\end{equation}
Note that by cross differentiation the perturbation $\varphi$ of the shock front can be excluded from the boundary conditions \eqref{lbc}:
\begin{equation}\label{lbc1}
\partial_{\star}v_2=a_0\partial_2p\quad \mbox{on}\ x_1=0.
\end{equation}

We now prove a ``linearized'' version of Proposition \ref{p1}.

\begin{proposition}
Suppose that problem \eqref{lsm}--\eqref{lindat} has a smooth solution $(U,\varphi)$ for $t\in [0,T]$.  Then, if the initial data \eqref{lindat} satisfy
\begin{equation}\label{8l}
{\rm div}\,F_j +(\mathcal{F}_j\cdot\nabla p)=0\quad \mbox{for}\ x \in\mathbb{R}_2^+ \quad (j=1,2),
\end{equation}
then \eqref{8l} holds for all $t\in [0,T]$.
\label{p2}
\end{proposition}

\begin{proof}
It follows from \eqref{ls1}, \eqref{ls3} and \eqref{ls4} that
\begin{equation}\label{gj}
{L}\mathcal{G}_j=0\quad \mbox{for}\ x \in\mathbb{R}_2^+ \quad (j=1,2),
\end{equation}
where $\mathcal{G}_j={\rm div}\,F_j +(\mathcal{F}_j\cdot\nabla p)$. Using the boundary conditions \eqref{lbc} and the interior equations \eqref{ls1}--\eqref{ls4} evaluated on $x_1=0$ and omitting detailed computations, we obtain
\begin{equation}\label{gjb}
 \mathcal{G}_j|_{x_1=0}=0, \quad j=1,2.
\end{equation}
Since $\mathcal{G}_j|_{t=0}=0$, problem \eqref{gj}, \eqref{gjb} has only solution $\mathcal{G}_j=0$.
\end{proof}

\section{Energy method for general deformations}
\label{s4}

We follow the idea of the energy method proposed by Blokhin \cite{Bl79} (see also \cite{BThand}) for gas dynamical shock waves. This idea is based on a symmetrization of the wave equation for $p$ (the perturbation of the pressure behind the shock) and the derivation of suitable boundary conditions for the obtained symmetric hyperbolic system for the second-order derivatives of $p$. From system \eqref{ls1}--\eqref{ls4} it is also possible to deduce a separate second-order equation for $p$. Indeed, applying the operator $M^2{L}$ to \eqref{ls1} and the div operator to \eqref{ls2}, after subtracting the results and using \eqref{8l} we obtain the hyperbolic equation
\[
\left(M^2{L}^2- \Delta -(\mathcal{F}_1\cdot\nabla )^2-(\mathcal{F}_2\cdot\nabla )^2\right)p=0\quad \mbox{for}\ x \in\mathbb{R}_2^+
\]
whose canonical form is the wave equation
\begin{equation}\label{we}
(L_1^2-L_2^2-L_3^2)p=0\quad \mbox{for}\ x \in\mathbb{R}_2^+
\end{equation}
in terms of the differential operators
\[
L_1=\frac{MM_*}{\beta^2}\,\partial_t-\frac{M\ell_0}{\beta^2M_*}\,\partial_2,\quad L_2=\partial_1-\frac{M^2}{\beta^2}\,\partial_t+\frac{\ell_0}{\beta^2}\,\partial_2,\quad
L_3=\frac{\sigma}{\beta M_*}\,\partial_2,
\]
where
\begin{equation}
\sigma = \sqrt{M_*^2(1+M_2^2)-\ell_0^2}=\sqrt{M_*^2+M_2^2+(\det \mathcal{F})^2}\quad \mbox{and}\quad M_2=\sqrt{\mathcal{F}_{21}^2+\mathcal{F}_{22}^2}.
\label{sigma}
\end{equation}

As in \cite{BThand}, we now symmetrize the wave equation \eqref{we}, i.e., we write down a symmetric hyperbolic system following from \eqref{we}. One can easily check that a sufficiently smooth solution of \eqref{we} satisfies the symmetric system
\begin{equation}
(B_0L_1-B_1L_2-B_2L_3)W =0 \quad\mbox{for}\ x \in\mathbb{R}_2^+\label{3.53}
\end{equation}
for the vector
\[
W=\begin{pmatrix} W_1 \\ W_2 \\  W_3\end{pmatrix}=\begin{pmatrix} L_1\widetilde{\nabla}p \\ L_2\widetilde{\nabla}p \\ L_3\widetilde{\nabla}p\end{pmatrix},
\]
where $\widetilde{{\nabla}}=(L_1,L_2,L_3)$,
\[
B_0=\left( \begin{array}{ccc} {\cal K} & {\cal L} & {{\cal M}} \\ {\cal L} & {\cal K} &
{\cal N} \\ {{\cal M}} & -{\cal N} & {\cal K} \end{array} \right),\quad B_1=\left(
\begin{array}{ccc} {\cal L} & {\cal K} & {\cal N} \\ {\cal K} & {\cal L} & {{\cal M}}
\\ -{\cal N} & {{\cal M}} & -{\cal L}
\end{array} \right),\quad
B_2=\left( \begin{array}{ccc} {{\cal M}} & -{\cal N} & {\cal K} \\ {\cal N} & -{{\cal
M}} & {\cal L} \\ {\cal K} & {\cal L} & {{\cal M}} \end{array} \right),
\]
and ${\cal K}$, ${\cal L}$, ${\cal M}$ and ${\cal N}$ are arbitrary symmetric matrices of order 3 whose final choice will be made below. Moreover, the matrices $B_0$, $B_1$ and $B_2$ can be written as follows:
\begin{equation}
\begin{split}
& B_0={\cal T}^{\top}\{ {I}_2{\otimes}{\cal H}\} {\cal T},\quad B_1={\cal
T}^{\top}\left\{ \left(
\begin{array}{cc} 0&-1\\-1&0 \end{array} \right){\otimes}{\cal H}\right\} {\cal T},\\
& B_2={\cal T}^{\top}\left\{ \left( \begin{array}{cc} -1&0\\0&1 \end{array}
\right){\otimes}{\cal H}\right\} {\cal T},
\end{split}\label{3.54}
\end{equation}
with
\begin{equation}
{\cal T}=\frac{1}{\sqrt{2}}\left( \begin{array}{ccc} 1&0&-1\\0&-1&0\\0&-1&0\\ 1&0&1
\end{array} \right){\otimes}{I}_3,\quad {\cal H}= \left( \begin{array}{cc} {\cal
K}-{{\cal M}} & -{\cal L}-{\cal N} \\ -{\cal L}+{\cal N} & {\cal K}+{{\cal M}}
\end{array} \right).
\label{H}
\end{equation}

Returning in \eqref{3.53} to $\partial_t$, $\partial_1$ and $\partial_2$, one gets the system
\begin{equation}
\widetilde{B}_0\partial_tW -B_1\partial_1W-\widetilde{B}_2\partial_2W =0\quad \mbox{for}\ x \in\mathbb{R}_2^+, \label{3.55}
\end{equation}
where
\[
\widetilde{B}_0=\frac{M}{\beta^2}(M_*B_0+MB_1)\quad\mbox{and}\quad \widetilde{B}_2=\frac{\sigma}{\beta M_*}B_2+\frac{M\ell_0}{\beta^2M_*}B_0+\frac{\ell_0}{\beta^2}B_1.
\]
In view of \eqref{3.54},
\[
\widetilde{B}_0=\frac{M}{\beta^2}{\cal T}^{\top}\left\{ \left( \begin{array}{cc} M_*&-M\\
     -M&M_*\end{array} \right){\otimes}{\cal H}\right\} {\cal T}.\label{3.55'}
\]
Thanks to the Lax condition $M<M_*$, cf. \eqref{Mach}, we have
\[
\left( \begin{array}{cc} M_*&-M\\
     -M&M_*\end{array} \right) >0
\]
and, hence, $\widetilde{B}_0>0$ provided that $\mathcal{H}>0$. That is, system \eqref{3.55} is symmetric hyperbolic under the condition $\mathcal{H}>0$.

We now deduce boundary conditions for system \eqref{3.55}.
Applying the vector differential operator $(M^2\partial_{\star},-\partial_{\star},0)$ to system \eqref{ls1}, \eqref{ls2}, evaluating the result on $x_1=0$ and using the boundary conditions \eqref{lbc}, \eqref{lbc1} and \eqref{gjb}, after some algebra we obtain the following second-order boundary condition for $p$:
\begin{equation}\label{bcp}
\begin{split}
M^2(1+d_0)\partial_t^2p -\beta^2\partial_t\partial_1p & +\left\{a_0\Big(M^2-M_1^2+\frac{M_2^2}{R}\Big)+\frac{2\ell_0^2}{M^2}\right\}\partial_2^2p\\[3pt] &
+\frac{\beta^2\ell_0}{M^2}\partial_1\partial_2p-\ell_0\Big(3+d_0+\frac{a_0}{R}\Big)\partial_t\partial_2p=0\quad\mbox{on}\ x_1=0.
\end{split}
\end{equation}
Using \eqref{we} evaluated on $x_1=0$ and omitting long but straightforward calculations, we rewrite \eqref{bcp} in terms of the differential operators $L_1$, $L_2$ and $L_3$:
\begin{equation}\label{bcp'}
L_1L_2p-M\widetilde{d}_0L_2^2p-\frac{M}{\beta^2}a_1L_3^2p+a_2(ML_1-M_*L_2)L_3p=0\quad\mbox{on}\ x_1=0,
\end{equation}
where
\[
\widetilde{d}_0=\frac{d_0}{M_*},\quad a_1=\beta^2\widetilde{d}_0 +a_0\left(M^2-M_1^2+\frac{M_2^2}{R} \right)\frac{M_*^3}{\sigma^2}+
a_2^2M_*\left(\beta^2+\frac{M^2}{2}\right),\quad a_2=\frac{\ell_0\beta}{M_*M\sigma}.
\]

Again following \cite{BThand}, we complete \eqref{bcp'} by the wave equation \eqref{we} evaluated on $x_1=0$ and some trivial relation:
\[
\left\{
\begin{array}{l}
L_1^2p-L_2^2p-L_3^2p=0,\\ L_3L_2p-L_2L_3p=0,\\[3pt]
{\displaystyle L_1L_2p-M\widetilde{d}_0L_2^2p-\frac{M}{\beta^2}a_1L_3^2p+a_2(ML_1-M_*L_2)L_3p=0\quad\mbox{on}\ x_1=0.}
\end{array}
\right.
\]
The last system can be written in the matrix form
\begin{equation}
{\cal A}W_1+{\cal B}W_2+{\cal C}W_3=0\quad\mbox{on}\ x_1=0,
\label{3.56}
\end{equation}
where
\[
{\cal A}=\left( \begin{array}{ccc} 1 & {\alpha}& 0 \\ 0 &    0   & 0 \\ 0 & 1 & Ma_2
\end{array} \right), \quad {\cal B}=\left( \begin{array}{ccc} -{\alpha} & -1 & 0 \\
0  &  0 &-1 \\0 & -M\widetilde{d}_0 &-M_*a_2\end{array} \right),\quad
{\cal C}=\left( \begin{array}{ccc} 0 & 0 & -1 \\ 0 & 1 &    0    \\[2pt] 0 & 0 &
{\displaystyle -\frac{Ma_1}{{\beta}^2}} \end{array} \right),
\]
and ${\alpha} >1$ is some constant. For $a_2=0$ the structure of the matrices $\mathcal{A}$, $\mathcal{B}$ and $\mathcal{C}$ coincides with that in gas dynamics \cite{BThand}.

Let the 3-vectors $V_k$ ($k=\overline{1,4}$) be corresponding vector components of the vector $V\in\mathbb{R}^{12}$ defined as follows:
\[
V=\left( \begin{array}{c} V_{\rm I}
    \\ V_{\rm II} \end{array} \right) = {\cal T}V,\quad
V_{\rm I}=\left( \begin{array}{c}V_1 \\
V_2
\end{array} \right)\;,\quad V_{\rm II}=\left( \begin{array}{c}
V_3 \\ V_4 \end{array} \right).
\]
Since
\[
W_1=\frac{\sqrt{2}}{2}(V_1+
V_4)\;,\quad W_2=-\sqrt{2}V_2=
-\sqrt{2}V_3,\quad
W_3=\frac{\sqrt{2}}{2}(V_4-V_1),
\]
the boundary conditions (\ref{3.56}) can also be written as
\begin{equation}
V_{\rm I}=GV_{\rm II}\quad\mbox{on}\ x_1=0, \label{3.57}
\end{equation}
with
\begin{equation}
G=\left( \begin{array}{cc} G_1 & -G_2 \\ I_3 & 0 \end{array} \right)\; ,\quad G_1=2({\cal
A}-{\cal C})^{-1}{\cal B},\quad G_2=({\cal A}-{\cal C})^{-1} ({\cal A}+{\cal C}).
\label{G}
\end{equation}

Assuming that $\mathcal{H}>0$ (i.e., $\widetilde{B}_0>0$) and applying standard arguments of the energy method, for the symmetric hyperbolic system \eqref{3.55} we obtain  the energy identity
\begin{equation}
I(t)+\int\limits_{0}^{t}\int\limits_{{\mathbb R}} (B_1W \cdot W )|_{x_1=0}\,{\rm d} x_2 {\rm d}s=I(0)
\label{3.60}
\end{equation}
for $t\in [0,T]$ and $W\in C([0,T], L^2({\mathbb R}^2_+))$, with
\[
I(t)=\int\limits_{{\mathbb R}^2_+} (\widetilde{B}_0{W}\cdot {W}){\rm d} x.
\]
In view of  (\ref{3.54}) and (\ref{3.57}),
\[
(B_1{W}\cdot{W})|_{x_1=0}=(G_0V_{\rm II}\cdot V_{\rm II})|_{x_1=0},
\]
where
\begin{equation}
-G_0=G^{\top}{\cal H}+{\cal H}G. \label{3.58}
\end{equation}

We now consider \eqref{3.58} as the Lyapunov equation \cite{Bell} for finding $\mathcal{H}$ (recall that the symmetric matrices ${\cal K}$, ${\cal L}$, ${\cal M}$ and ${\cal N}$ appearing in the definition of $\mathcal{H}$ in \eqref{H} are arbitrary yet). As is known, if all the eigenvalues of the matrix $G$ lie strictly in the open left-half complex plane, then for any real symmetric positive definite matrix $G_0$ this equation has a unique real solution $\mathcal{H}$ which is again a symmetric matrix. Assume that the matrix $G$ in \eqref{G} has the mentioned property of its eigenvalues $\lambda _j(G)$:
\begin{equation}
\Real\lambda_j(G)<0\quad\mbox{for all}\  j=\overline{1,6}.
\label{eG}
\end{equation}
Referring the reader to \ref{appA} for technical computations, here we just write down the following necessary and sufficient condition for the fulfilment of property \eqref{eG}:
\begin{equation}
(M_*^2+M^2)\sigma^2-\left\{R(M^2-M_1^2)+M_2^2 \right\}M_*^4+\ell_0^2(2\beta^2+M^2)-2|\ell_0|\beta M\sigma >0.
\label{usc}
\end{equation}

Assuming that the unperturbed flow satisfies condition \eqref{usc}, we find the real symmetric matrix
\[
{\cal H}=\left( \begin{array}{cc} {\cal H}_1 & {\cal H}_2 \\ {\cal H}_2^{\,\top} & {\cal
H}_3 \end{array} \right) >0
\]
which is a unique solution of the Lyapunov equation \eqref{3.58}, where $\mathcal{H}_1$ and $\mathcal{H}_3$ are symmetric matrices. Having in hand the matrix $\mathcal{H}$, we then define the symmetric matrices ${\cal K}$, ${\cal L}$, ${\cal M}$ and ${\cal N}$:
\[
{\cal K}=\frac{1}{2}({\cal H}_1+{\cal H}_3), \quad {{{\cal M}}}=\frac{1}{2}({\cal
H}_3-{\cal H}_1),\quad
{\cal L}=-\frac{1}{2}({\cal H}_2+{\cal H}_2^{\,\top}), \quad {\cal N}=\frac{1}{2}({\cal
H}_2^{\,\top}-{\cal H}_2).
\]

We underline that if we consider equation \eqref{3.58} as an equation for finding $\mathcal{H}$ for a given matrix $G_0$, then the matrix $G_0$ is still an arbitrary real symmetric positive definite matrix. By changing $G_0$, we change the solution $\mathcal{H}$ and, hence, our choice of ${\cal K}$, ${\cal L}$, ${\cal M}$ and ${\cal N}$. At the present moment our important assumption is that the matrix $G_0>0$ and we will below choose how big should be its norm. Thanks to this assumption $(B_1{W}\cdot{W})|_{x_1=0}>0$. Moreover, since
\[
V_{\rm II}=\frac{\sqrt{2}}{2}\left(\begin{array}{c} -W_2 \\
W_1+W_3
\end{array} \right),
\]
then
\begin{equation}
\begin{split}
(B_1W\cdot W)|_{x_1=0} >C_1\big( & (L_1^2p)^2+(L_1L_2
p)^2+(L_1L_3p)^2 \\  & +(L_2^2p)^2+(L_2 L_3 p)^2+(L_3^2
p)^2 \big)\big|_{x_1=0}> C_2{\cal P}|_{x_1=0},
\end{split}
\label{3.61'}
\end{equation}
where
\[
{\cal
P}=(\partial_t^2p)^2+(\partial_t\partial_1p)^2+(\partial_t\partial_2p)^2+(\partial_1^2p)^2+(\partial_1\partial_2p)^2+(\partial_2^2p)^2
\]
is the sum of all second-order derivatives of $p$, and $C_{1}=C_{1}(G_0)>0$ and $C_{2}=C_{2}(G_0)>0$ are constants depending on the norm of the matrix $G_0$.

In fact, \eqref{3.61'} means that the boundary conditions \eqref{3.57} are {\it strictly dissipative}. Using this, from \eqref{3.60} we deduce the estimate
\begin{equation}
I(t)+C_2\int\limits_{0}^{t}\int\limits_{{\mathbb R}} {\cal P}|_{x_1=0}\,{\rm d} x_2 {\rm d}s\leq I(0)
\label{3.60c}
\end{equation}
giving us a control on not only the solution $W$ but also on its trace $W|_{x_1=0}$. However, the most important thing is that thanks to this strict dissipativity we can obtain the strict dissipativity of the boundary conditions for system \eqref{lsm} prolonged up to second-order derivatives of $U$ (see below).

For sufficiently smooth solutions, from system \eqref{lsm} we obtain the following system for all the second-order derivatives of $U$:
\begin{equation}\label{lsmp}
\mathfrak{A}_0\partial_t\mathcal{U}+\mathfrak{A}_1\partial_1\mathcal{U}+\mathfrak{A}_2\partial_2\mathcal{U}=0\quad\mbox{for}\ x\in \mathbb{R}^2_+,
\end{equation}
where
\[
\mathcal{U}=(\partial_t^2U,\partial_t\partial_1U,\partial_t\partial_2U,\partial_1^2U,\partial_1\partial_2U,\partial_2^2U)\quad\mbox{and}\quad
\mathfrak{A}_k=I_6\otimes \mathcal{A}_k\quad (k=0,1,2).
\]
The energy identity for \eqref{lsmp} reads
\begin{equation}
\mathcal{I}(t)-\int\limits_{0}^{t}\int\limits_{{\mathbb R}} (\mathfrak{A}_1\mathcal{U} \cdot \mathcal{U} )|_{x_1=0}\,{\rm d} x_2 {\rm d}s =\mathcal{I}(0),
\label{3.60p}
\end{equation}
with
\[
\mathcal{I}(t)=\int\limits_{{\mathbb R}^2_+} (\mathfrak{A}_0\mathcal{U}\cdot \mathcal{U}){\rm d} x.
\]

As in gas dynamics \cite{Bl79,BThand}, one can check that all the traces of the second-order derivatives of the components of $U$ can be expressed through the sum of the traces of the second-order derivatives of $p$. Indeed, using the boundary conditions \eqref{lbc} as well as the system \eqref{lsm} and its derivatives evaluated on $x_1=0$, taking into account that for shock waves the boundary matrix $\mathcal{A}_1$ is not singular, and omitting detailed arguments, we come to the inequality
\begin{equation}\label{PP}
(\mathfrak{A}_1\mathcal{U} \cdot \mathcal{U} )|_{x_1=0} \leq C_3{\cal P}|_{x_1=0},
\end{equation}
where $C_3>0$ is a constant depending on the coefficients of our linearized problem. It follows from \eqref{3.60p} and \eqref{PP} that
\begin{equation}
\mathcal{I}(t)-C_3\int\limits_{0}^{t}\int\limits_{{\mathbb R}} {\cal P}|_{x_1=0}\,{\rm d} x_2{\rm d}s \leq\mathcal{I}(0),
\label{3.60p'}
\end{equation}
and summing up \eqref{3.60c} and \eqref{3.60p'} gives
\begin{equation}
I(t)+\mathcal{I}(t)+C_4\int\limits_{0}^{t}\int\limits_{{\mathbb R}} {\cal P}|_{x_1=0}\,{\rm d} x_2{\rm d}s \leq I(0)+\mathcal{I}(0),
\label{3.60f}
\end{equation}
where the constant $C_4= C_2-C_3>0$ thanks to the choice of the matrix $G_0$ with a sufficiently big norm.

Combining \eqref{3.60f} with the elementary inequality
\[
\| U (t)\|^2_{L^2(\mathbb{R}^2_+)} \leq \| U \|^2_{L^2([0,t]\times \mathbb{R}^2_+)} +\| \partial_tU \|^2_{L^2([0,t]\times \mathbb{R}^2_+)} +
\| U (0)\|^2_{L^2(\mathbb{R}^2_+)}
\]
and using the positive definiteness of the matrices $\widetilde{B}_0$ and $\mathfrak{A}_0$, we get the energy inequality
\begin{equation}
\nt U(t)\nt^2_{H^2(\mathbb{R}^2_+)}+\int\limits_{0}^{t}\int\limits_{{\mathbb R}} {\cal P}|_{x_1=0}\,{\rm d} x_2{\rm d}s\leq C\Bigg\{ \| U_0\|^2_{H^2(\mathbb{R}^2_+)}+\int\limits_{0}^{t}\nt U(s)\nt^2_{H^2(\mathbb{R}^2_+)} {\rm d}s\Bigg\},
\label{enin}
\end{equation}
where
\[
\nt U(t)\nt^2_{H^2(\mathbb{R}^2_+)}:=\sum\limits_{j=0}^2\|\partial_t^jU(t)\|^2_{H^{2-j}(\mathbb{R}^2_+)}.
\]
Here and below $C$ is a positive constant that can change from line to line. Throwing away the positive boundary integral in the left-hand side of \eqref{enin} and applying then Gronwall's lemma, we obtain the energy a priori estimate
\[
\nt U(t)\nt_{H^2(\mathbb{R}^2_+)}\leq C\| U_0\|_{H^2(\mathbb{R}^2_+)}
\]
for $t\in [0,T]$.

Since we have a control on the trace $\mathcal{P}|_{x_1=0}$ in \eqref{enin} and since, as was noted above, we can express $|\mathcal{U}|_{x_1=0}|^2$ through $\mathcal{P}|_{x_1=0}$, we can finally derive the following a priori estimate (we omit simple arguments which, in particular, include the usage of the energy inequality \eqref{enin}, the trace theorem,  etc.):
\begin{equation}
\| U\|_{H^2([0,T]\times\mathbb{R}^2_+)}+ \| U_{|x_1=0}\|_{H^2([0,T]\times\mathbb{R})}\leq C\| U_0\|_{H^2(\mathbb{R}^2_+)},
\label{aprest}
\end{equation}
where $C$ depends on $T$. Moreover, since we have a control on the trace of the solution in \eqref{aprest}, by using the second and third boundary conditions in \eqref{lbc} (recall that $R\neq 1$, see Remark \ref{r1}), exactly as for gas dynamical shock waves in \cite{Bl79,BThand}, we can also estimate the front perturbation $\varphi$:
\[
\| U\|_{H^2([0,T]\times\mathbb{R}^2_+)}+ \| U_{|x_1=0}\|_{H^2([0,T]\times\mathbb{R})} + \| \varphi\|_{H^3([0,T]\times\mathbb{R})}
\leq C\left\{\| U_0\|_{H^2(\mathbb{R}^2_+)} +\| \varphi_0\|_{L^2(\mathbb{R})}\right\}.
\]

We have actually constructed a strictly dissipative 2-symmetrizer \cite{Tsiam} for problem \eqref{lsm}--\eqref{lindat}. In fact, referring to \cite{Tsiam} (or just directly revising arguments above), we can also write down an a priori estimate for the corresponding inhomogeneous problem, i.e., problem \eqref{lsm}--\eqref{lindat} with a given source term $f(t,x)\in \mathbb{R}^7$ in the right-hand side of the interior equations \eqref{lsm} and a given source term $g(t,x_2)\in \mathbb{R}^7$ in the right-hand side of the boundary conditions \eqref{lbc}. This estimate reads
\begin{equation}
\begin{split}
\| U\|_{H^2([0,T]\times\mathbb{R}^2_+)}+ & \| U_{|x_1=0}\|_{H^2([0,T]\times\mathbb{R})} + \| \varphi\|_{H^3([0,T]\times\mathbb{R})}\\ &
\leq C\left\{\| U_0\|_{H^2(\mathbb{R}^2_+)} +\| \varphi_0\|_{L^2(\mathbb{R})} +\| f\|_{H^2([0,T]\times\mathbb{R}^2_+)}
+\| g\|_{H^2([0,T]\times\mathbb{R})}\right\},
\end{split}
\label{aprest'}
\end{equation}
where $C>0$ depends on $T$ and the constant coefficients of the problem (the parameters of the unperturbed flow) and does not depend on the initial data and the source terms. Since estimate \eqref{aprest'} is an a priori estimate {\it without loss of derivatives} from the initial data and the source terms, the energy method above can be considered an indirect proof of the uniform Kreiss--Lopatinski condition \cite{Kreiss}, and condition \eqref{usc} is sufficient for uniform stability, i.e., sufficient for the fulfilment of the uniform Kreiss--Lopatinski condition. We have thus obtained the following stability result for shock waves in 2D elastodynamics.

\begin{theorem}
Let a rectilinear shock wave satisfies the Lax conditions \eqref{Mach} and \eqref{Mach-}. Let also it satisfies condition \eqref{usc}, i.e.,
\begin{equation}
\begin{split}
\big(1+&\mathcal{F}_{11}^2+\mathcal{F}_{12}^2+M^2\big)\left(1+(\mathcal{F}:\mathcal{F})+(\det \mathcal{F})^2\right) \\ & -\left\{R(M^2-\mathcal{F}_{11}^2-\mathcal{F}_{12}^2)+\mathcal{F}_{21}^2+\mathcal{F}_{22}^2 \right\}\left(1+\mathcal{F}_{11}^2+\mathcal{F}_{12}^2\right)^2 \\ &
+\left({\cal F}_{11}{\cal F}_{21}+{\cal F}_{12}{\cal F}_{22}\right)^2\left\{2(1+\mathcal{F}_{11}^2+\mathcal{F}_{12}^2)-M^2\right\} \\
&\qquad >2M\left|{\cal F}_{11}{\cal F}_{21}+{\cal F}_{12}{\cal F}_{22}\right|\sqrt{\left(1+\mathcal{F}_{11}^2+\mathcal{F}_{12}^2-M^2\right)
 \left(1+(\mathcal{F}:\mathcal{F})+(\det \mathcal{F})^2\right)} \,,
\end{split}
\label{usc'}
\end{equation}
where $M$ is the downstream Mach number, $R$ measures the competition between downstream and upstream densities and $\mathcal{F}=(\mathcal{F}_{ij})_{i,j=1,2}$ is the scaled deformation gradient behind of the shock. Then the a priori estimate \eqref{aprest'} holds for solutions of the corresponding inhomogeneous linearized problem  and the rectilinear shock wave is uniformly stable.
\label{t1}
\end{theorem}

According to the results in \cite{M1,M2,Met} and their extension to hyperbolic symmetrizable systems with characteristics of variable multiplicities \cite{Kwon,MZ}, all uniformly stable shocks are {\it structurally stable}. Roughly speaking (we do not discuss regularity, compatibility conditions, etc.), this means that if the uniform Kreiss--Lopatinski condition holds at each point of the initial shock, then this shock exists locally in time. In other words, as soon as planar (or rectilinear for the 2D case) shock waves are uniformly stable according to the linear analysis with constant coefficients, we can make the conclusion about structural stability of corresponding curved shocks.
Strictly speaking, if uniform stability was established by spectral analysis, i.e., by the direct test of the uniform Kreiss--Lopatinski condition, then for the deduction of a priori estimates by Kreiss' symmetrizers technique \cite{Kreiss} one needs to check the fulfilment of either Majda\-'s block structure condition \cite{M1} or the hypotheses of a general variable-multiplicity stability framework introduced by M\'etivier and Zumbrun \cite{MZ}.

Since we have already derived an a priori estimate without loss of derivatives and since this a priori estimate \eqref{aprest'} was obtained by the construction of a strictly dissipative 2-symmetrizer \cite{Tsiam}, we do not need to check structural conditions from \cite{M1,MZ}. Referring to \cite{Tsiam}, we get the structural stability of shock waves for which condition \eqref{usc'} holds at each point of the initial shock front. However, we guess that, as the system of ideal compressible isentropic or non-isentropic MHD (see \cite{MZ,Kwon}), the system of compressible elastodynamics satisfies the conditions introduced by M\'etivier and Zumbrun \cite{MZ}. The proof of this is postponed to future research.

Setting formally $\mathcal{F}=0$ in \eqref{usc'}, we obtain the uniform stability condition
\begin{equation}
M^2(R-1)<1
\label{gas}
\end{equation}
found by Majda \cite{M1} (and written in our notations) for shock waves in isentropic gas dynamics. More precisely, condition \eqref{gas} should be also completed by the Lax conditions $M<1$  and $M_->1$ (they are  \eqref{Mach} and \eqref{Mach-} for $\mathcal{F}=0$). Condition \eqref{gas} is necessary and sufficient for uniform stability of Lax shock waves in isentropic gas dynamics. In the next section (see Proposition \ref{p3}), by the direct test of the uniform Kreiss--Lopatinski condition we prove that our stability condition \eqref{usc'} is not only sufficient but also necessary for uniform stability for the particular deformations: the case of {\it stretching} $\mathcal{F}_{12}=\mathcal{F}_{21}=0$ and the ``opposite'' case $\mathcal{F}_{11}=\mathcal{F}_{22}=0$ (for these cases $\ell_0={\cal F}_{11}{\cal F}_{21}+{\cal F}_{12}{\cal F}_{22}=0$).

For stretching, the stability condition \eqref{usc'} becomes
\begin{equation}\label{str}
1+{\cal F}_{11}^2 +M^2 -R(1+{\cal F}_{11}^2)(M^2-{\cal F}_{11}^2)+{\cal F}_{22}^2M^2>0.
\end{equation}
Introducing the ``elastic'' Mach number $\widetilde{M}=\sqrt{M^2-{\cal F}_{11}^2}>0$, we rewrite \eqref{str} as
\begin{equation}\label{str'}
\widetilde{M}^2(R-1)<1+\frac{{\cal F}_{11}^2(1-\widetilde{M}^2)+{\cal F}_{22}^2(\widetilde{M}^2+{\cal F}_{11}^2)}{1+{\cal F}_{11}^2}.
\end{equation}
As the Mach number in gas dynamics, the ``elastic'' Mach number $\widetilde{M}$ satisfies $0<\widetilde{M}<1$, cf. \eqref{Mach}. Moreover, the fraction in the right-hand side of \eqref{str'} is strictly positive ($\det\mathcal{F}={\cal F}_{11}{\cal F}_{22}\neq 0$). Therefore, comparing \eqref{gas} and \eqref{str'}, we see that inequality \eqref{str'} for  $\widetilde{M}$ is less restrictive than inequality \eqref{gas} for $M$. In this sense, we make the conclusion that the elastic force plays a {\it stabilizing role} in the stability of shock wave that is presumably clear from the physical point of view. Clearly, the same is true for the particular deformation with $\mathcal{F}_{11}=\mathcal{F}_{22}=0$ for which condition \eqref{usc'} reads
\begin{equation}\label{opp}
1+{\cal F}_{12}^2 +M^2 -R(1+{\cal F}_{12}^2)(M^2-{\cal F}_{12}^2)+{\cal F}_{21}^2M^2>0
\end{equation}
(\eqref{opp} coincides with  \eqref{str} if we replace ${\cal F}_{11}$ and ${\cal F}_{22}$ with ${\cal F}_{12}$ and ${\cal F}_{21}$ respectively).

In fact, even {\it for general deformations} one can show that the elastic force plays a stabilizing role. Since all compressive shocks in isentropic gas dynamics were proved by Majda  \cite{M1} to be uniformly stable for convex equations of state $p=p(\rho )$, it is almost obvious that the same is true for shock waves in elastodynamics. We first get the following auxiliary result.

\begin{lemma}
Let $p(\rho )$ be a convex function of $\rho$. Then all compressive shock waves satisfy the Lax conditions \eqref{Mach} and \eqref{Mach-} as well as  the  ``elastic'' counterpart
\begin{equation}
\widetilde{M}^2(R-1)<1
\label{gas'}
\end{equation}
of condition \eqref{gas}, where $\widetilde{M}=\sqrt{M^2-{\cal F}_{11}^2-{\cal F}_{12}^2}$ is the ``elastic'' Mach number ($0<\widetilde{M}<1$).
\label{l1}
\end{lemma}

\begin{proof}
We rewrite the second jump condition in \eqref{sbc} as
\begin{equation}\label{R}
Rw = \frac{[p(\hat{\rho})]}{p'(\hat{\rho}^+)[\hat{\rho}]},
\end{equation}
where $w=M^2-{\cal F}_{11}^2-{\cal F}_{12}^2$,  $[p(\hat{\rho})]=p(\hat{\rho}^+)-p(\hat{\rho}^-)$, and we now consider $p$ as a function of $\rho$. Since
$p(\rho )$ is convex,
\[
\frac{[p(\hat{\rho})]}{p'(\hat{\rho}^+)[\hat{\rho}]}\leq 1.
\]
Then, it follows from \eqref{R} that $Rw\leq 1$, and for compressive shocks ($R>1$) this implies $w< 1$. The second condition in \eqref{sbc} also reads
\[
wM_-^2=M^2R\frac{[p(\hat{\rho})]}{p'(\hat{\rho}^-)[\hat{\rho}]}.
\]
Thanks to the convexity of $p(\rho )$,
\[
\frac{[p(\hat{\rho})]}{p'(\hat{\rho}^-)[\hat{\rho}]}\geq 1.
\]
That is, $wM_-^2\geq M^2R>M^2>0$. This means that the Lax conditions $M>M_1$ (i.e., $w>0$) and \eqref{Mach-} (i.e., $M_-\sqrt{w}>M$) are satisfied. We may thus consider $\widetilde{M}=\sqrt{w} =\sqrt{M^2-{\cal F}_{11}^2-{\cal F}_{12}^2}$ as the ``elastic'' Mach number for general deformations, and we have proved that $0<\widetilde{M}<1$. At last, rewriting the  inequality $R\widetilde{M}^2\leq 1$ obtained above as
\[
\widetilde{M}^2(R-1)\leq 1-\frac{1}{R},
\]
we get \eqref{gas'} because $1-\frac{1}{R}<1$.
\end{proof}

Using Lemma \ref{l1}, we are now ready to prove the following theorem.

\begin{theorem}
All compressive shock waves in isentropic 2D elastodynamics with a convex equation of state $p=p(\rho )$, in particular, with the equation of state $p=A\rho^{\gamma}$ ($A>0$, $\gamma >1$) of a polytropic fluid  are structurally stable.
\label{t2}
\end{theorem}

The proof of Theorem \ref{t2} is given in \ref{appB} and connected with a long chain of elementary manipulations with inequalities. In the very beginning of this proof we see that if the Lax conditions \eqref{Mach} and \eqref{Mach-} hold and a certain value
$\mathcal{D}$ is strictly positive (see \eqref{usc1} and \eqref{DD}), then the elastic force plays a stabilizing role. The rest of the proof is a {\it nontrivial} check that $\mathcal{D}>0$. Referring then to Lemma \ref{l1}, we conclude  that the stability condition \eqref{usc'} always holds for compressive shocks if the equation of state is convex. At last, as was already noted above, the reference to \cite{Tsiam} gives structural stability.

\section{Spectral analysis for particular deformations}
\label{s5}

For the reader's convenience, we write down here our linear constant coefficients stability problem \eqref{lsm}, \eqref{lbc} by taking into account the fact that the front perturbation $\varphi$ can be excluded from the boundary conditions \eqref{lbc} (see \eqref{lbc1}):
\begin{eqnarray}
 \mathcal{A}_0\partial_tU+\mathcal{A}_1\partial_1U+\mathcal{A}_2\partial_2U =0 & &\mbox{for}\ x\in \mathbb{R}^2_+, \label{71}\\
 \mathfrak{B}_0\partial_tU+\mathfrak{B}_2\partial_2U+\mathfrak{B}U =0 & &\mbox{on}\ x_1=0. \label{72}
\end{eqnarray}
where \eqref{72} is the matrix form of the boundary conditions
\[
\begin{split}
 v_1 +d_0p -\frac{\ell_0}{M^2R}\,v_2=0,\quad \partial_{\star}v_2-a_0\partial_2p=0,\quad
F_{11}+{\cal F}_{11}\,p-\frac{{\cal F}_{21}}{R}\,v_2=0, &  \\[6pt]
 F_{12}+{\cal F}_{12}\,p-\frac{{\cal F}_{22}}{R}\,v_2=0,\quad
F_{21}-{\cal F}_{11}\,v_2=0, \quad   F_{22}-{\cal F}_{12}\,v_2=0 & \quad\ \mbox{on}\ x_1=0
\end{split}
\]
(cf. \eqref{lbc}, \eqref{lbc1}), and the matrices $\mathfrak{B}_0$, $\mathfrak{B}_2$ and  $\mathfrak{B}$ of order $6\times 7$ can be easily written down. We should also complete \eqref{71}, \eqref{72} by initial data.

Applying the Fourier-Laplace transform to \eqref{71}, \eqref{72} (the Fourier transform with respect to $x_2$ and
the Laplace transform with respect to $t$), we obtain the following boundary-value
problem for a system of ODEs:
\begin{align}
& \frac{{\rm d}\widetilde{U}}{{\rm d} x_1}=\mathfrak{A}(s,{\omega} )\widetilde{ U}, \quad
 x_1>0,\label{3.1}\\
& (s\mathfrak{B}_0+i\omega\mathfrak{B}_2+\mathfrak{B})\widetilde{ U}|_{x_1=0}=0,\label{3.2}
\end{align}
where $\widetilde{U}=\widetilde{U}(x_1,s,{\omega})$ is the Fourier-Laplace transform
of $U (t,x )$;
\[
s=\eta +{i}\xi ,\quad \eta >0,\quad (\xi ,{\omega})\in \mathbb{R}^2,\quad \mathfrak{A}=\mathfrak{A}(s,\omega )=-\mathcal{A}_1^{-1}(s\mathcal{A}_0+{i}\omega \mathcal{A}_2).
\]
The eigenvalues $\lambda$ of the matrix $\mathfrak{A}$ obey the dispersion relation
\begin{equation}
\det (s\mathcal{A}_0+\lambda \mathcal{A}_1+{i}\omega \mathcal{A}_2)=0. \label{3.4}
\end{equation}
Since our shock waves are 1-shocks, the boundary matrix $\mathcal{A}_1$ has only one positive eigenvalue. Then, in view of Hersh's lemma \cite{Hersh}, for all $\eta >0$ and $(\xi ,{\omega})\in \mathbb{R}^2$ equation \eqref{3.4} has a {\it unique} solution $\lambda =\lambda^+=\lambda^+ (\eta ,\xi ,\omega )$ lying strictly in the open right-half complex plane ($\Real\lambda >0$).

We write down the following homogenous system of linear algebraic equations:
\begin{align}
& (s\mathcal{A}_0+\lambda^+ \mathcal{A}_1+{i}\omega \mathcal{A}_2)X =0\,,\label{3.5}
\\
& \big(\mathcal{A}_1\widetilde{U}_0\big)\cdot X=0,
\label{3.6}
\end{align}
where $\widetilde{U}_0=\widetilde{U}(0,s,\omega )$ is a vector satisfying the boundary conditions \eqref{3.2}. Since $\lambda^+$ is a simple eigenvalue, we can choose six linearly independent equations of system \eqref{3.5}. Adding them to equation \eqref{3.6}, we obtain for the vector $X$ the linear system
\begin{equation}
\mathfrak{L}X =0
\label{X}
\end{equation}
whose determinant is, in fact, the Lopatinski determinant (see below).
Following \cite{Tcmp}, we are now ready to give definitions of the Lopatinski condition and the uniform Lopatinski condition for problem \eqref{71}, \eqref{72}.  Such kind of definitions which can be used for hyperbolic problems having the 1-shock property \cite{Tcmp} are equivalent to Kreiss' classical  definitions \cite{Kreiss}.

\begin{definition}
Problem \eqref{71}, \eqref{72} satisfies the Lopatinski condition if the Lopatinski determinant $\det \mathfrak{L} (\eta
,\xi ,\omega ,\lambda^+ ) \neq 0$ for all $\eta >0$ and  $(\xi , \omega )\in
\mathbb{R}^2$.
\label{d3.2}
\end{definition}

\begin{definition}
Problem \eqref{71}, \eqref{72} satisfies the uniform Lopatinski condition if the Lopatinski determinant $\det \mathfrak{L} (\eta
,\xi ,\omega ,\lambda^+ ) \neq 0$ for all $\eta \geq 0$ and  $(\xi , \omega )\in
\mathbb{R}^2$ (with $\eta^2+\xi^2+{\omega}^2\neq 0$), where $\lambda^+ (0,\xi ,\omega )= \lim\limits_{\eta \rightarrow +0}
\lambda^+ (\eta ,\xi ,\omega )  $.
\label{d3.3}
\end{definition}

We first obtain the dispersion relation \eqref{3.4} (even for general deformations). Omitting straightforward calculations, we get the following polynomial equation for finding the unique $\lambda=\lambda^+$:
\begin{equation}
\Omega^3\left( M^2\Omega^2-\sigma_1^2-\sigma_2^2\right)\left( M^2\Omega^2-\sigma_1^2-\sigma_2^2-\lambda^2+\omega^2\right)=0,
\label{3.4'}
\end{equation}
with $\Omega =s+\lambda$, $\sigma_1=\mathcal{F}_{11}\lambda+i\omega\mathcal{F}_{21}$ and $\sigma_2=\mathcal{F}_{12}\lambda+i\omega\mathcal{F}_{22}$. Clearly, the solution $\lambda =-s$ of the equation $\Omega =0$ is not the desired $\lambda=\lambda^+$. But, actually, the same is true for the solutions of the equation
\[
M^2\Omega^2-\sigma_1^2-\sigma_2^2=0.
\]
Indeed, for $\omega =0$ its solutions are $\lambda=-Ms/(M- M_1)$ and $\lambda=-Ms/(M+ M_1)$. Referring again to Hersh's lemma \cite{Hersh}, we conclude that the last equation has no root $\lambda =\lambda^+$ also for all $\omega\neq 0$ (and $\eta >0$). That is, $\lambda^+$ is one of the two roots of the equation
\begin{equation}
 M^2\Omega^2-\sigma_1^2-\sigma_2^2-\lambda^2+\omega^2=0.
\label{3.4''}
\end{equation}

There is no technical problem to write down the algebraic system \eqref{X} for general deformations. However, for general deformations it is unfortunately impossible to study the equations $\det \mathfrak{L}=0$  and \eqref{3.4''} {\it analytically}, i.e., without computer calculations like those in \cite{Tcmp} which make sense only for concrete parameters of the unperturbed flow, in our case, for concrete numerical values of six dimensionless parameters $M$, $\mathcal{F}_{11}$, $\mathcal{F}_{12}$, $\mathcal{F}_{21}$, $\mathcal{F}_{22}$ and $R$. This is why from now on we restrict ourself to the case of stretching $\mathcal{F}_{12}=\mathcal{F}_{21}=0$. Then, the vector $\widetilde{U}_0$ being determined up to a nonzero factor can be taken in the form
\[
\widetilde{U}_0=\Big(s,-d_0s,ia_0\omega,-\mathcal{F}_{11}s,ia_0\mathcal{F}_{11}\omega, \frac{ia_0\mathcal{F}_{22}\omega}{R},0\Big),
\]
and we find that
\[
\mathcal{A}_1\widetilde{U}_0=\Big( (1-d_0)s,(1-M^2d_0+\mathcal{F}_{11}^2)s,ia_0(M^2-\mathcal{F}_{11}^2)\omega , (d_0-1)\mathcal{F}_{11}s,
0,\frac{ia_0\mathcal{F}_{22}\omega}{R},0\Big).
\]

Replacing the first line of the matrix $s\mathcal{A}_0+\lambda^+ \mathcal{A}_1+{i}\omega \mathcal{A}_2$ with the vector $\mathcal{A}_1\widetilde{U}_0$ considered  as a row-vector, we get the matrix $\mathfrak{L}$ (see \eqref{3.5}, \eqref{3.6}). Omitting technical calculations, we find that
\[
\det \mathfrak{L}=\frac{\beta^2\Omega^2(\omega^2-\lambda^2)}{2M^2}\big\{ (\lambda^2-\omega^2)s +M^2\Omega\lambda s +\mathcal{F}_{11}^2\lambda^2s + \mathcal{F}_{22}^2\omega^2\lambda +R(M^2-\mathcal{F}_{11}^2)\omega^2\Omega \big\},
\]
where $\lambda =\lambda^+$, i.e., $\lambda$ should be the solution of equation \eqref{3.4''} with the property $\Real\lambda >0$ for $\eta >0$. Since
$\Omega^2(\omega^2-\lambda^2)\neq 0$ for $\lambda =\lambda^+$, the equality $\det \mathfrak{L}=0$ is equivalent to
\begin{equation}
(\lambda^2-\omega^2)s +M^2\Omega\lambda s +\mathcal{F}_{11}^2\lambda^2s + \mathcal{F}_{22}^2\omega^2\lambda +R(M^2-\mathcal{F}_{11}^2)\omega^2\Omega =0.
\label{3.4'''}
\end{equation}
The test of the (uniform) Kreiss--Lopatinski condition is thus reduced to the study of solutions $(s,\lambda )$ of  system \eqref{3.4''}, \eqref{3.4'''} for all real $\omega$.

For $\omega =0$ it follows from \eqref{3.4''} that $\lambda=\lambda^+=Ms/(M_*-M)$ whereas \eqref{3.4'''} reads
\[
\lambda s\big(M^2s + (M^2+M_*^2)\lambda\big)=0.
\]
Substituting the above $\lambda^+$ into the last equation, we find the unique solution $s=0$. Hence, the Lopatinski condition holds whereas the uniform Lopatinski condition is also satisfied because the solution $s=0$ is prohibited by the requirement $\eta^2+\xi^2\neq 0$ (for $\omega =0$, see Definition \ref{d3.3}). That is, we may assume that $\omega \neq 0$. Moreover, since the left-hand sides in \eqref{3.4''} and \eqref{3.4'''} are homogeneous functions of $s$, $\lambda$ and $\omega$, without loss of generality we can suppose that $\omega =1$. Then system \eqref{3.4''}, \eqref{3.4'''} reads
\begin{align}
& M^2\Omega^2 -M_*^2\lambda^2+K_2=0, \label{1.}\\
& \Omega (M^2\lambda s +K) +(M_*^2\lambda^2  -K_2)s=0 ,\label{2.}
\end{align}
where $K=R(M^2-\mathcal{F}_{11}^2)+\mathcal{F}_{22}^2 >0$ and $K_2=1+\mathcal{F}_{22}^2$ and  (recall that for stretching $M_*=\sqrt{1+\mathcal{F}_{11}^2}$ and $\mathcal{F}_{11}^2<M^2 <M_*^2$, see \eqref{Mach}).

It follows from \eqref{1.} that $M_*^2\lambda^2-K_2 =M^2\Omega^2$. Substituting this into \eqref{2.}, we obtain
\begin{equation}
\Omega (M^2\Omega^2-M^2\lambda^2+K)=0.
\label{0.}
\end{equation}
Since $\Omega\neq 0$ for $\lambda =\lambda^+$ and $\eta >0$, \eqref{0.} is reduced to
\begin{equation}
M^2\Omega^2-M^2\lambda^2+K=0.
\label{3.}
\end{equation}
We can thus consider \eqref{1.}, \eqref{3.} instead of system \eqref{1.}, \eqref{2.}. Considering \eqref{1.}, \eqref{3.} as a linear system for $\Omega^2$ and $\lambda^2$ (recall that $\Omega= s+\lambda$), we find
\begin{equation}
\Omega^2= \frac{M_*^2}{M^2\widehat{M}^2} (K_1-K)\quad \mbox{and}\quad \lambda^2=\frac{1}{\widehat{M}^2}(K_2-K),
\label{4.}
\end{equation}
where $\widehat{M}=\sqrt{M_*^2-M^2}$ and $K_1=M^2K_2/M_*^2$. In view of \eqref{Mach}, $\widehat{M}>0$ and $K_1<K_2$.

We now analyse the behavior of the roots $(s,\lambda )$ of \eqref{4.} depending on the position of a point $K$ in comparison with that of the points $K_1$ and $K_2$ on the $K$-axis.  If $K\leq K_1$, then both $\Omega$ and $\lambda$ are real, and the requirement $\lambda>0$ implies
\[
\eta=
s=\Omega -\lambda =\frac{M_*}{M\widehat{M}}\sqrt{K_1-K}-\frac{1}{\widehat{M}}\sqrt{K_2-K}=\frac{1}{\widehat{M}}\left( \sqrt{K_2-K\frac{M_*^2}{M^2}}- \sqrt{K_2-K}\right) <0.
\]
That is, for $K\leq K_1$ not only the Lopatinski condition but also the uniform Lopatinski  condition holds (see Definitions \ref{d3.2} and \ref{d3.3}). In other words, the inequality $K\leq K_1$ describes a part of the uniform stability domain.

If $K_1<K<K_2$, then $\lambda$ is real and $\Real\Omega =0$. Taking $\lambda>0$, we have
$\Real\Omega = \eta +\lambda =0$ implying that $\eta =-\lambda <0$. Therefore, again not only the Lopatinski condition but also the uniform Lopatinski  condition holds. If $K\geq K_2$, then $\Real\lambda =\Real\Omega =0$ and, hence, $\Real s=0$, i.e., the Lopatinski condition holds. That is, we can already make the conclusion that no ill-posendess happens and for $K<K_2$  shock waves are uniformly stable. It remains to understand whether neutral stability (violation of the uniform Lopatinski condition) may happen for $K\geq K_2$.

Recall that $\lambda=\lambda^+$ is the solution of \eqref{1.} with $\Real \lambda >0$ for $\eta >0$. The dispersion relation \eqref{1.} has two roots
\[
\lambda^{\pm}=\frac{M^2s\pm\sqrt{M^2M_*^2s^2+K_2\widehat{M}^2}}{\widehat{M}^2}.
\]
Let $\lambda^{\pm}_{|\eta =+0}=i\delta^{\pm}$. Then
\[
\frac{\delta^{\pm}}{\xi}=\frac{1}{\widehat{M}^2}\left(M^2\pm\sqrt{M^2M_*^2-\frac{K_2\widehat{M}^2}{\xi^2}}\,\right)
\]
(we note that system \eqref{1.}, \eqref{3.} has no roots with $s=0$ and, hence, $\xi\neq 0$ for $\eta=+0$; recall that $s=\eta +i\xi$). We have
\begin{equation}\label{delta}
\frac{\delta^+}{\xi}\geq \frac{M^2}{{\widehat{M}}^2}
\end{equation}
that is equivalent to
\begin{equation}\label{delta+}
\frac{\xi}{\delta^+}\leq \frac{{\widehat{M}}^2}{M^2}.
\end{equation}

The fulfilment of the uniform Lopatinski condition is equivalent to the absence of solutions of system \eqref{1.}, \eqref{3.} for $\eta =+0$ corresponding to  $\delta =\delta^+$. The latter is equivalent to the violation of \eqref{delta+}, i.e., the roots $(s,\lambda )=(\eta +i\xi ,{\rm Re}\,\lambda +i\delta )$ of \eqref{1.}, \eqref{3.} should satisfy the inequality
\[
\frac{\xi}{\delta}>\frac{{\widehat{M}}^2}{M^2}\quad \mbox{for}\quad \eta=+0 .
\]
The last inequality is rewritten as
\begin{equation}
\frac{\xi}{\delta}+1>\left(\frac{M_*}{M}\right)^2.
\label{5.}
\end{equation}

From  \eqref{1.} and \eqref{3.} we find
\[
\left.\frac{\Omega^2}{\lambda^2}\right|_{\eta =0}=\left( \frac{\xi}{\delta}+1\right)^2=\frac{M_*^2(K-K_1)}{M^2(K-K_2)}>0\quad \mbox{for}\quad K>K_2.
\]
Then, for $K> K_2$ inequality \eqref{5.} is equivalent to the inequality
\[
\frac{M_*^2(K-K_1)}{M^2(K-K_2)}>\left(\frac{M_*}{M}\right)^4
\]
which, in turn, can be shown to be equivalent to the inequality
\begin{equation}
K<K_1+K_2.
\label{6.}
\end{equation}

If $K=K_2$, then it follows from \eqref{4.} that
\[
\lambda =0\quad\mbox{and}\quad s^2=\frac{M_*^2}{M^2\widehat{M}^2} (K_1-K_2)<0.
\]
That is, for $K=K_2$ we have $\eta =0$ and $\lambda =0$. In particular, $\delta =0$, but we know that $\delta^+\neq 0$ (cf. \eqref{delta}). Hence, ${\rm Im}\,\lambda =\delta^-$ for $K=K_2$. This means that for $K=K_2$ the uniform Lopatinski condition holds. Summarizing the above, we make the conclusion that the uniform Lopatinski condition holds if and only if inequality \eqref{6.} is satisfied.

\begin{remark}
{\rm
The transition  $K=K_1+K_2$ from uniform stability to weak stability could be also found by analyzing the coincidence of the eigenvalues $\lambda ^+$ and $\lambda^-$ for $\eta =+0$. Indeed, using \eqref{4.}, one can show that $\delta^+=\delta^-=M^2\xi/\widehat{M}^2$ if only $K=K_1+K_2$. The interested reader can find more information about generic types of transitions in \cite{BRSZ}.
}
\end{remark}

One can show that \eqref{6.} is equivalent to condition \eqref{str} which is our structural stability condition \eqref{usc'} for $\mathcal{F}_{12}=\mathcal{F}_{21}=0$ found by the energy method. For the particular case $\mathcal{F}_{11}=\mathcal{F}_{22}=0$ all the arguments of spectral analysis are totally ``symmetric'' to the case of stretching and we obtain condition \eqref{opp} necessary and sufficient for  uniform stability (in the above arguments we should just replace ${\cal F}_{11}$ and ${\cal F}_{22}$ with ${\cal F}_{12}$ and ${\cal F}_{21}$ respectively). We have thus proved the following proposition.

\begin{proposition}
For the particular cases $\mathcal{F}_{12}=\mathcal{F}_{21}=0$ and $\mathcal{F}_{11}=\mathcal{F}_{22}=0$ the Lopatinski condition always holds, i.e., all rectilinear shock waves are, at least, weakly stable. Moreover, conditions  \eqref{str} and \eqref{opp} are necessary and sufficient for the uniform stability of rectilinear shocks for the particular cases $\mathcal{F}_{12}=\mathcal{F}_{21}=0$ and $\mathcal{F}_{11}=\mathcal{F}_{22}=0$ respectively.
\label{p3}
\end{proposition}

\paragraph{Acknowledgements} A part of this work was done during the stay  of Yuri Trakhinin  as a Visiting Professor of the University of Brescia in April--May 2018. This author gratefully thanks  the Mathematical Division of the Department of Civil, Environmental, Architectural Engineering and Mathematics  of the University of Brescia for its kind hospitality. This work was also partially supported by RFBR (Russian Foundation for Basic Research) grant No. 19-01-00261-a and FFABR (Fondo di Finanziamento per le Attivit\`a Base di Ricerca) grant No. D83C18000060001.

\appendix
\section{Uniform stability condition}\label{appA}

By virtue of \eqref{G}, the eigenvalues $\lambda$ of the matrix $G$ obey the equation
\[
\det \left( (\mathcal{A}-\mathcal{C})\lambda^2-2\mathcal{B}\lambda +\mathcal{A}+\mathcal{C}\right)=0.
\]
Omitting simple calculations, we obtain from it the following sixth-order polynomial equation:
\[
(\lambda^2+2\alpha\lambda +1) h(\lambda )=0,
\]
where $h(\lambda )=b_4\lambda^4+b_3\lambda^3+b_2\lambda^2+b_1\lambda+b_0$, with
\[
\begin{split}
& b_4= M\left(\frac{a_1}{\beta^2}+a_2\right),\quad b_3=2(1+M_*a_2),\quad b_2=\frac{2M}{\beta^2}\left(2\beta^2\widetilde{d}_0-a_1\right),\\
& b_1=2(1-M_*a_2),\quad b_0=M\left(\frac{a_1}{\beta^2}-a_2\right).
\end{split}
\]

Thanks to our assumption that $\alpha >1$ the both roots of the equation $\lambda^2+2\alpha\lambda +1=0$ are strictly negative. That is, condition \eqref{eG} is reduced to the requirement that all the roots of the fourth-order polynomial equation $h(\lambda )=0$ have negative real parts. Referring to the the Li\'{e}nard--Chipart criterion \cite{LS}, this is equivalent to the following set of six inequalities:
\begin{equation}
b_k>0\quad\mbox{for}\ k=\overline{0,4}\quad \mbox{and}\quad b_1(b_2b_3-b_1b_4)-b_3^2b_0>0.
\label{bk}
\end{equation}
The inequalities $b_1>0$ and $b_3>0$ are equivalent to the single inequality $a_2^2<1/M_*^2$ which is reduced to $\ell_0^2<M^2(1+M^2_2)$. Thanks to the Lax condition $M>M_1$, cf. \eqref{Mach}, even the more restrictive requirement $\ell_0^2<M^2M^2_2$ always holds:
\[
M^2M_2^2>M_1^2M_2^2=\ell_0^2+(\det\mathcal{F})^2>\ell_0^2.
\]
The inequalities $b_1>0$ and $b_3>0$ are thus satisfied automatically.

Let us now analyze the last inequality in \eqref{bk}. Omitting calculations, we rewrite it as
\begin{equation}
\frac{\beta^2\widetilde{d}_0-a_1}{\beta^2}-\frac{M_*\beta^2}{2M^2}\,a_2^2>0.
\label{bkl}
\end{equation}
For the fulfilment of \eqref{bkl} it is necessary that $\beta\widetilde{d}_0-a_1>0$. Since $\widetilde{d}_0>0$, the fulfilment of \eqref{bkl} guarantees that $b_2>0$. After some algebra we reduce inequality \eqref{bkl} to
\[
\ell_0^2<M^2M_2^2+RM^2(M^2-M_1^2).
\]
The last inequality always holds thanks to the Lax condition $M>M_1$ and the above mentioned true inequality  $\ell_0^2<M^2M^2_2$.

It remains to require the fulfilment of the conditions $b_0>0$ and $b_4>0$. They are reduced to the single inequality $a_1>\beta^2|a_2|$. This inequality multiplied by $2M^2M_*\sigma^2/\beta^2$ is nothing else as the sufficient uniform stability condition \eqref{usc}.

\section{Proof of Theorem \ref{t2}}\label{appB}

Taking into account Lemma \ref{l1}, we may introduce the ``elastic'' Mach number $\widetilde{M}=\sqrt{M^2-M_1^2}$, and we know that $0<\widetilde{M}<1$. Omitting technical calculations,
we equivalently rewrite the stability condition \eqref{usc} as
\begin{equation}\label{usc1}
\widetilde{M}^2(R-1)<1+ \frac{\mathcal{D}}{M_*^4},
\end{equation}
where
\begin{equation}\label{DD}
\mathcal{D}=(M\sigma -|\ell_0|\beta -M_*^2\widetilde{M})(M\sigma -|\ell_0|\beta +M_*^2\widetilde{M}).
\end{equation}
Comparing \eqref{usc1} with Majda's condition \eqref{gas}, we see that if $\mathcal{D}>0$, then the elastic force plays stabilizing role. Moreover, in view of \eqref{gas'}, $\mathcal{D}>0$ implies that the stability condition \eqref{usc1} always holds for compressive shocks  with a convex equation of state.

That is, our goal now is to prove that $\mathcal{D}>0$. We calculate:
\begin{equation}
\label{RR}
M^2\sigma^2-\ell_0^2\beta^2=M_*^2\big(M_2^2\widetilde{M}^2+M^2+(\det \mathcal{F})^2\big)>0.
\end{equation}
This implies that $M\sigma -|\ell_0|\beta >0$ and, hence,
\[
M\sigma -|\ell_0|\beta +M_*^2\widetilde{M}>0
\]
(recall that $\sigma >0$ and $\beta =\sqrt{1-\widetilde{M}^2}>0$). That is, the condition  $\mathcal{D}>0$ is equivalent to
\begin{equation}
M\sigma -|\ell_0|\beta - M_*^2\widetilde{M}>0.
\label{D}
\end{equation}

Inequality  \eqref{D} is rewritten as the following desired inequality for the ``elastic'' Mach number $\widetilde{M}$:
\begin{equation}
\sigma\sqrt{M_1^2+\widetilde{M}^2} -M_*^2\widetilde{M} >|\ell_0|\sqrt{1-\widetilde{M}^2}.
\label{D'}
\end{equation}
Since
\[
\sigma^2 \big(M_1^2+\widetilde{M}^2\big) -M_*^4\widetilde{M}^2=M_*^2M_1^2 \big(1-\widetilde{M}^2\big)+\big(M_2^2+(\det \mathcal{F})^2\big)M^2>0
\]
(we omit simple intermediate calculations here), the left-hand side in \eqref{D'} is strictly positive. Squaring \eqref{D'}, we obtain the equivalent inequality
\begin{equation}
\big(\sigma^2+M_*^4+\ell_0^2\big)\widetilde{M}^2 +\sigma^2M_1^2-\ell_0^2> 2M_*^2\sigma\widetilde{M}\sqrt{M_1^2+\widetilde{M}^2} .
\label{D''}
\end{equation}
Setting in \eqref{RR} formally $M=M_1$, we get the true inequality $\sigma^2M_1^2-\ell_0^2>0$. Hence, the left-hand side in \eqref{D''} is strictly positive. Squaring \eqref{D''} gives the equivalent inequality
\begin{equation}
\begin{split}
\Big\{\big(\sigma^2+ & M_*^4+\ell_0^2\big)^2 -4M_*^4\sigma^2\Big\} Z^2 \\ & +2 \left\{\big(\sigma^2+M_*^4+\ell_0^2\big) \big(\sigma^2M_1^2-\ell_0^2\big)-2M_*^4\sigma^2M_1^2\right\}Z+\big(\sigma^2M_1^2-\ell_0^2\big)^2 >0
\end{split}
\label{D1}
\end{equation}
for $Z=\widetilde{M}^2 \in (0,1)$.

Noticing that
\[
\sigma^2+  M_*^4+\ell_0^2=M_*^2\big(1+M_*^2+M_2^2\big)\quad \mbox{and}\quad \sigma^2M_1^2-\ell_0^2=M_*^2\big(M_1^2+M_1^2M_2^2-\ell_0^2\big),
\]
dividing \eqref{D1} by $M_*^4$ and using that $\ell_0^2=M_1^2M_2^2-(\det \mathcal{F})^2$, we equivalently rewrite \eqref{D1} as
\begin{equation}
\begin{split}
\big\{\big(M_1^2+M_2^2\big)^2 & -4\varkappa^2\big\}  Z^2 \\ & +2 \left\{\big(M_2^2-M_1^2\big)\big(M_1^2+\varkappa^2\big) +2\varkappa^2-2M_1^2M_2^2\right\} Z +\big(M_1^2+\varkappa^2\big)^2 >0,
\end{split}
\label{D2}
\end{equation}
where $\varkappa =\det\mathcal{F}$. That is, it remains to prove that inequality \eqref{D2} holds true for all $Z\in (0,1)$. Standard arguments for the quadratic function in \eqref{D2} lead to a bulky inequality whose validity is unclear, but fortunately we can just write down \eqref{D2} as
\[
4\varkappa^2Z(1-Z) +\left( (M_1^2+M_2^2)Z-(M_1^2+\varkappa^2)\right)^2 +4M_2^2\varkappa^2Z>0.
\]
The last inequality is true for all $Z\in (0,1)$. We have thus proved that $\mathcal{D}>0$. It means that if the function $p(\rho )$ is convex, then the a priori estimate \eqref{aprest'} always holds for compressive shocks. Referring to \cite{Tsiam}, we conclude that these shock waves are structurally stable.

\end{document}